\documentclass[12 pt]{amsart}
\usepackage{amssymb,latexsym,amsmath,amscd,amsthm,graphicx, color}
\usepackage[all]{xy}
\usepackage{pgf,tikz}
\usepackage{mathrsfs}
\usepackage{cite}
\usetikzlibrary{arrows}
\definecolor{uuuuuu}{rgb}{0.26666666666666666,0.26666666666666666,0.26666666666666666}
\definecolor{xdxdff}{rgb}{0.49019607843137253,0.49019607843137253,1.}
\definecolor{ffqqqq}{rgb}{1.,0.,0.}

\definecolor{uuuuuu}{rgb}{0.26666666666666666,0.26666666666666666,0.26666666666666666}
\definecolor{qqwuqq}{rgb}{0.,0.39215686274509803,0.}
\definecolor{zzttqq}{rgb}{0.6,0.2,0.}
\definecolor{xdxdff}{rgb}{0.49019607843137253,0.49019607843137253,1.}
\definecolor{qqqqff}{rgb}{0.,0.,1.}
\definecolor{cqcqcq}{rgb}{0.7529411764705882,0.7529411764705882,0.7529411764705882}

\theoremstyle{plain}

\newtheorem{theo1}[subsubsection]{Theorem}

\newtheorem{lemma}[subsection]{Lemma}
\newtheorem{lemma1}[subsubsection]{Lemma}

\newtheorem{prop}[subsection]{Proposition}
\newtheorem{propo}[subsection]{Proposition}

\theoremstyle{definition}
\newtheorem{defi}[subsection]{Definition}
\newtheorem{algorithm}[subsection]{Alogorithm}
\newtheorem{cor}[subsection]{Corollary}

\newtheorem{rem}[subsection]{Remark}

\newtheorem{note}[subsection]{Note}

\newcommand{\uu}{\cup}% union
\newcommand{\ii}{\cap}% intersection
\newcommand{\sci}{\subset}% strictly contained in
\newcommand{\es}{\emptyset}% the empty set
\newcommand{\set}[1]{\{#1\}}% set
\newcommand{\ga}{\alpha}
\newcommand{\gb}{\beta}

\newcommand{\gd}{\delta}
\newcommand{\gs}{\sigma}
\newcommand{\gt}{\tau}

\newcommand{\tbf}{\textbf}% text bold
\newcommand{\tit}{\textit}% text italic

\newcommand{\D}[1]{\mathbb{#1}}% Doubled - blackboard bold - only caps, uas as \D{A}
\newcommand{\te}{\text}% same as \mathrm command.

\newcommand{\ol}{\overline}
\newcommand{\ul}{\underline}
\newcommand{\nd}{\noindent}

\begin{document}

To appear, Real Analysis Exchange
\title{Optimal Quantization For Mixed Distributions}
%%\title{An Overview of Optimal Quantization for Mixed Distributions}

\author{Mrinal Kanti Roychowdhury}
\address{School of Mathematical and Statistical Sciences\\
University of Texas Rio Grande Valley\\
1201 West University Drive\\
Edinburg, TX 78539-2999, USA.}
\email{mrinal.roychowdhury@utrgv.edu}

\thanks{}

\date{}
\maketitle

\pagestyle{myheadings}\markboth{Mrinal Kanti Roychowdhury}{Optimal Quantization for Mixed Distributions}

\begin{abstract} The basic goal of quantization for probability distribution is to reduce the number of values, which is typically uncountable, describing a probability distribution to some finite set and thus approximation of a continuous probability distribution by a discrete distribution. Mixed distributions are an exciting new area for optimal quantization. In this paper, we have determined the optimal sets of $n$-means, the $n$th quantization errors, and the quantization dimensions of different mixed distributions. Besides, we have discussed whether the quantization coefficients for the mixed distributions exist. The results in this paper will give a  motivation and insight into more general problems in quantization for mixed distributions.

\end{abstract}

\section{Introduction}
The most common form of quantization is rounding-off. Its purpose is to reduce the cardinality of the representation space, in particular, when the input data is real-valued. It has broad application in engineering and technology (see \cite{GG, GN, Z}).  Let $\D R^d$ denote the $d$-dimensional Euclidean space equipped with the Euclidean norm $\|\cdot\|$, and let $P$ be a Borel probability measure on $\D R^d$ where $d\geq 1$. Then, the $n$th \textit{quantization
error} for $P$, with respect to the squared Euclidean distance, is defined by
\begin{equation*} \label{eq1} V_n:=V_n(P)=\inf \Big\{V(P; \ga) : \ga \subset \mathbb R^d, 1\leq \text{ card}(\ga) \leq n \Big\},\end{equation*}
where $V(P; \ga)= \int \min_{a\in\alpha} \|x-a\|^2 dP(x)$ represents the distortion error due to the set $\ga$ with respect to the probability distribution $P$.
A set $\ga\sci \D R^d$ is called an optimal set of $n$-means for $P$ if $V_n(P)=V(P; \ga)$. Of course, such a set $\ga$ exists if the mean squared error or the expected squared Euclidean distance $\int \| x\|^2 dP(x)$ is finite (see \cite{AW, GKL, GL, GL1}). For a continuous Borel probability measure an optimal set of $n$-means contains exactly $n$ elements (see \cite{GL1}). The elements of an optimal set of $n$-means are called \tit{optimal quantizers}.  For some work in this direction, one can see \cite{CR, DR, GL2, R1, R2, R3, R4, R5}.
For a finite set $\ga \sci \D R^d$ and $a\in \ga$, by $M(a|\ga)$ we denote the set of all elements in $\D R^d$ which are nearest to $a$ among all the elements in $\ga$.
$M(a|\ga)$ is called the \tit{Voronoi region} generated by $a\in \ga$. On the other hand, the set $\set{M(a|\ga) : a \in \ga}$ is called the \tit{Voronoi diagram} or \tit{Voronoi tessellation} of $\D R^d$ with respect to the set $\ga$. The point $a$ is called the centroid of its own Voronoi region if $a=E(X : X\in M(a|\ga))$, where $X$ is a $P$-distributed random variable.
A Borel measurable partition $\set{A_a : a\in \ga}$ is called a \tit{Voronoi partition} of $\D R^d$ with respect to the probability distribution $P$, if $P$-almost surely  $A_a\sci M(a|\ga)$ for all $a\in \ga$.
Let us now state the following proposition (see \cite{GG, GL1}).
\begin{prop} \label{prop0}
Let $\ga$ be an optimal set of $n$-means, $a \in \ga$, and $M (a|\ga)$ be the Voronoi region generated by $a\in \ga$. Then, for every $a \in\ga$,
$(i)$ $P(M(a|\ga))>0$, $(ii)$ $ P(\partial M(a|\ga))=0$, $(iii)$ $a=E(X : X \in M(a|\ga))$, and $(iv)$ $P$-almost surely the set $\set{M(a|\ga) : a \in \ga}$ forms a Voronoi partition of $\D R^d$.
\end{prop}
The above proposition implies that the points in an optimal set are the centroids of their own Voronoi regions, in other words, the points in an optimal set are an evenly-spaced distribution of sites in the domain with minimum distortion error with respect to a given probability measure and is therefore very useful in many fields, such as clustering, data compression, optimal mesh generation, cellular biology, optimal quadrature, coverage control and geographical optimization, for more details one can see \cite{DFG, OBSC}. Besides, it has  applications in energy efficient distribution of base stations in a cellular network  \cite{HCHSVH, KKR, S}. In both geographical and cellular applications
the distribution of users is highly complex and often modeled by a fractal \cite{ABDHW, LZSC}.
The numbers
\[\ul D(P):=\liminf_{n\to \infty}  \frac{2\log n}{-\log V_n(P)} \te{ and } \ol D(P):=\limsup_{n\to \infty} \frac{2\log n}{-\log V_n(P)} \]
are, respectively, called the \tit{lower} and \tit{upper quantization dimensions} of the probability measure $P$. If $\ul D(P)=\ol D (P)$, then the common value is called the \tit{quantization dimension} of $P$ and is denoted by $D(P)$. For any $s\in (0, +\infty)$, the numbers $\liminf_n n^{\frac 2 s} V_n(P)$ and $\limsup_n n^{\frac 2s} V_n(P)$ are, respectively, called the $s$-dimensional \tit{lower} and \tit{upper quantization coefficients} for $P$. If the $s$-dimensional lower and upper quantization coefficients for $P$ are finite and positive, then $s$ coincides with the quantization dimension of $P$ (see \cite{GL1}).

By a probability vector $(p_1, p_2, \cdots, p_N)$ it is meant that $0<p_j<1$ for all $1\leq j\leq N$, and $\sum_{j=1}^N p_j=1$. We now give the following definition.

\begin{defi}
Let $P_1, P_2, \cdots, P_N$ be Borel probability measures on $\D R^d$, and $(p_1, p_2, \cdots, p_N)$ be a probability vector. Then, a Borel probability measure $P$ on $\D R^d$ is called a \tit{mixed probability distribution}, or in short, \tit{mixed distribution}, generated by $(P_1, P_2, \cdots, P_N)$ and the probability vector if for all Borel subsets $A$ of $\D R^d$, $P(A)=p_1P_1(A)+p_2P_2(A)+\cdots+p_NP_N(A)$. Such a mixed distribution is denoted by $P:=p_1P_1+p_2P_2+\cdots+p_NP_N$, and $P_1, P_2, \cdots, P_N$ are called the {\it components} of the mixed distribution.
\end{defi}

The following proposition follows from \cite[Theorem~2.1]{L}.

\begin{prop}  \label{prop1000}
Let $P$ be the mixed distribution generated by $(P_1, P_2)$ associated with the probability vector $(p, 1-p)$, i.e., $P=pP_1+(1-p)P_2$, where $0<p<1$. Assume that both $D(P_1)$ and $D(P_2)$ exist. Then, $D(P)=\max\set{D(P_1), D(P_2)}$.
\end{prop}
In this paper, our goal is to determine the optimal sets of $n$-means, the $n$th quantization errors for all positive integers $n$ and the quantization dimensions, and the quantization coefficients for different mixed distributions. In Section~\ref{sec1}, we have considered a mixed distribution $P:=pP_1+(1-p)P_2$, where $p=\frac 12$, $P_1$ is a uniform distribution on the closed interval $C:=[0, \frac 12]$, and $P_2$ is a discrete distribution on $D:=\set{\frac 23, \frac 56, 1}$. For this mixed distribution, in Subsection~\ref{sec2}, first we have determined the optimal sets of $n$-means and the $n$th quantization errors for $n=2, 3, 4, 5$, and then in Theorem~\ref{th61}, we give a general formula to determine the optimal sets of $n$-means and the $n$th quantization errors for all $n\geq 5$. In Proposition~\ref{prop612}, we further show that the quantization coefficient for this mixed distribution exists as a finite positive number yielding the fact that the quantization dimension for this mixed distribution exists and equals the dimension of the underlying space (see Remark~\ref{rem001}).

 In Section~\ref{sec3}, for a mixed distribution $P:=pP_1+(1-p)P_2$, where $P_1$ is an absolutely continuous probability measure supported by the closed interval $C:=[0, 1]$, and $P_2$ is discrete on $D:=\set{0, 1}$, we mention a rule how to determine the optimal sets of $n$-means. In Proposition~\ref{prop1111}, for a special case, we give a closed formula to determine the optimal sets of $n$-means and the $n$th quantization errors for all $n\geq 2$. As mentioned in Remark~\ref{rem4}, in Proposition~\ref{prop1112}, we have proved a claim that the optimal sets for a mixed distribution may not be unique.

 In Section~\ref{sec5}, we determine the optimal sets of $n$-means, and the $n$th quantization errors for all $n\geq 2$ for a mixed distribution $P:=\frac 12 P_1+\frac 12 P_2$, where $P_1$ is a Cantor distribution with support lying in the closed interval $[0, \frac 12]$, and $P_2$ is discrete on $D:=\set{\frac 23, \frac 56, 1}$. In this section, first we have determined the optimal sets of $n$-means and the $n$th quantization errors for $n=2, 3, 4, 5$, and then in Theorem~\ref{th611}, we give a general formula to determine them for all $n\geq 5$. In Remark~\ref{rem00012}, we show that the quantization dimension for $P$ exists, but the quantization coefficient does not exist. In Section~\ref{sec61}, we give some remarks about the optimal quantization for mixed distributions.

In Section~\ref{sec7}, we consider a mixed distribution $P:=\frac 12 P_1+\frac 12 P_2$, where both $P_1$ and $P_2$ are Cantor distributions. For this mixed distribution, we determine the optimal sets of $n$-means and the $n$th quantization errors for all $n\geq 2$. In Theorem~\ref{Th3}, we show that the quantization coefficient for $P$ does not exist. In Section~\ref{sec8}, we give some discussion and open problems. Finally we would like to mention that the mixed distributions are an exciting new area for optimal quantization, and the results in this paper will give a  motivation and insight into more general problems.

\section{Quantization with $P_1$ uniform and $P_2$ discrete} \label{sec1}
Let $P_1$ be a uniform distribution on the closed interval $C:=[0, \frac 12]$, i.e., $P_1$ is a probability distribution on $\D R$ with probability density function $g$ given by
 \[ g(x)=\left\{\begin{array}{ccc}
2  & \te{ if }x \in C,\\
 0  & \te{ otherwise}.
\end{array}\right.
\]
Let $P_2$ be a discrete probability distribution on $\D R$ with probability mass function $h$ given by $h(x)=\frac 13 $ for $x\in D$, and $h(x)=0$ for $x\in \D R\setminus D$, where $D:=\set{\frac 23, \frac 56, 1}$. Let $P$ be the mixed distribution on $\D R$ such that $P=\frac 12 P_1+\frac 12 P_2$. Notice that the support of $P_1$ is $C$, and the support of $P_2$ is $D$ implying that the support of $P$ is $C\uu D$. Thus, for a Borel subset $A$ of $\D R$, we can write
\[P(A)=\frac 12 P_1(A\ii C)+\frac 12 P_2(A\ii D).\]
We now prove the following lemma.

\begin{lemma}\label{lemma1}
Let $E(X)$ and $V:=V(X)$ represent the expected value and the variance of a random variable $X$ with distribution $P$. Then, $E(X)=\frac{13}{24}$ and $V=\frac{181}{1728}=0.104745$.
\end{lemma}
\nd\tbf{Proof:}  We have
\begin{align*} &E(X)=\int x dP=\frac 12 \int x dP_1+\frac 12 \int x dP_2=\frac 12 \int_{[0, \frac{1}{2}]} x g(x)dx+\frac 12 \sum_{x\in D} x h(x)=\frac{13}{24}, \te{ and } \\ &E(X^2)=\int x^2 dP=\frac 12 \int x^2 dP_1+\frac 12 \int x^2 dP_2=\frac 12 \int_{[0, \frac 12]}x^2 g(x)dx+\frac 12 \sum_{x\in D} x^2 h(x)=\frac{43}{108},
\end{align*}
implying $V:=V(X)=E(X^2)-(E(X))^2 =\frac{43}{108}-\left(\frac{13}{24}\right)^2=\frac{181}{1728}$. Thus, the lemma is yielded.
\qed

\begin{note}
Following the standard rule of probability, we see that $E\|X-a\|^2 =\int(x-a)^2 dP=V(X)+(a-E(X))^2=V+(a-\frac{13}{24})^2$, which yields the fact that the optimal set of one-mean consists of the expected value $\frac {13}{24}$, and the corresponding quantization error is the variance $V$ of the random variable $X$. By $P(\cdot|C)$, we denote the conditional probability measure on $C$, i.e., $P(\cdot|C)=\frac {P(\cdot\ii C)}{P(C)}$, in other words, for any Borel subset $B$ of $C$ we have $P(B|C)=\frac{P(B\ii C)}{P(C)}$. Notice that $P(\cdot|C)$ is a uniform distribution with density function $f$ given by
\[ f(x)=\left\{\begin{array}{ccc}
2  & \te{ if }x \in C,\\
 0  & \te{ otherwise},
\end{array}\right.
\]
implying the fact that $P(\cdot|C)=P_1$. Similarly, $P(\cdot|D)=P_2$. In the sequel, for $n\in \D N$ and $i=1, 2$, by $\ga_n(P_i)$ and $V_n(P_i)$, it is meant the optimal sets of $n$-means and the $n$th quantization error with respect to the probability distributions $P_i$. If nothing is mentioned within a parenthesis, i.e., by $\ga_n$ and $V_n$, it is meant an optimal set of $n$-means and the $n$th quantization error with respect to the mixed distribution $P$.
\end{note}

\begin{prop} \label{prop1111}
Let $P_1$ be the uniform distribution on the closed interval $[a, b]$ and $n\in \D N$. Then, the set $\set{a + \frac{(2i-1)(b-a)}{2n} : 1\leq i\leq n}$ is a unique optimal set of $n$-means for $P_1$, and the corresponding quantization error is given by $V_n(P_1)=\frac{(a-b)^2}{12 n^2}$.
\end{prop}
\nd\tbf{Proof:}  Notice that the probability density function $g$ of $P_1$ is given by  \[ g(x)=\left\{\begin{array}{ccc}
\frac 1{b-a}  & \te{ if }x \in [a, b],\\
 0  & \te{ otherwise}.
\end{array}\right.
\]
Since $P_1$ is uniformly distributed on $[a, b]$, the boundaries of the Voronoi regions of an optimal set of $n$-means will divide the interval $[a, b]$ into $n$ equal subintervals, i.e., the boundaries of the Voronoi regions are given by
\[\left\{a, \ a+\frac {(b-a)}{n}, \ a+ \frac {2(b-a)}{n},\ \cdots \ a+ \frac {(n-1)(b-a)}{n}, \ a+ \frac {n(b-a)} {n}\right\}.\]
This implies that an optimal set of $n$-means for $P_1$ is unique, and it consists of the midpoints of the boundaries of the Voronoi regions, i.e., the optimal set of $n$-means for $P_1$ is given by $\ga_n(P_1):=\set{a+\frac{(2i-1)(b-a)}{2n} : 1\leq i\leq n}$ for any $n\geq 1$. Then, the $n$th quantization error for $P_1$ due to the set $\ga_n(P_1)$ is given by
\begin{align*}
&V_n(P_1)=n \int_{[a, a+\frac {b-a}{n}]} \Big(x-(a+\frac {b-a}{2n})\Big)^2  dP_1=n \int_{[0, \frac 1{2n}]} \frac 1{b-a} (x-\frac 1{4n})^2  dx=\frac{(a-b)^2}{12 n^2},
\end{align*}
which yields the proposition.
\qed

\begin{cor} \label{cor1}
Let $P_1$ be the uniform distribution on the closed interval $[0, \frac 12]$ and $n\in \D N$. Then, the set $\set{\frac{2i-1}{4n} : 1\leq i\leq n}$ is a unique optimal set of $n$-means for $P_1$, and the corresponding quantization error is given by $V_n(P_1)=\frac 1{48n^2}$.
\end{cor}

\begin{rem}
Notice that if $\gb\sci \D R$, then
\begin{align*}
&\int \min_{b \in \gb}\|x-b\|^2 dP=\frac 12 \int_{[0, \frac 12]} \min_{b \in \gb}(x-b)^2 g(x) dx+\frac 12 \sum_{x\in D}  \min_{b \in \gb}(x-b)^2h(x), \te{ and so, }
\end{align*}
\begin{equation} \label{eq001} \int \min_{b \in \gb}\|x-b\|^2 dP=\int_{[0, \frac 12]} \min_{b \in \gb}(x-b)^2 dx+\frac 16 \sum_{x\in D}  \min_{b \in \gb}(x-b)^2. \end{equation}
\end{rem}

\subsection{Optimal sets of $n$-means and the errors for all $n\geq 2$} \label{sec2}

In this subsection, we first determine the optimal sets of $n$-means and the $n$th quantization error for the mixed distribution $P$. Then, we show that the quantization dimension of $P$ exists and equals the quantization dimension of $P_1$, which again equals one, which is the dimension of the underlying space. To determine the distortion error in this subsection we will frequently use equation~\eqref{eq001}.
\begin{lemma1}
Let $\ga$ be an optimal set of two-means. Then, $\ga=\set{\frac 14, \frac 56}$ with quantization error $V_2=\frac{17}{864}=0.0196759.$
\end{lemma1}
\nd\tbf{Proof:}  Consider the set of two-points $\gb$ given by $\gb:=\set{\frac 14, \frac 56}$. Then, the distortion error is
\[\int \min_{b \in \gb}\|x-b\|^2 dP= \int_{[0, \frac 12]}(x-\frac 14)^2 dx+\frac 16 \sum_{x\in D} (x-\frac 56)^2=\frac{17}{864}=0.0196759.\]
Since $V_2$ is the quantization error for two-means we have $V_2\leq 0.0196759$. Let $\ga:=\set{a_1, a_2}$ be an optimal set of two-means with $a_1<a_2$. Since the optimal points are the centroids of their own Voronoi regions, we have $0<a_1<a_2\leq 1$. If $\frac{13}{32}\leq a_1$, then
\[V_2\geq \int_{[0, \frac{13}{32}]} (x-\frac{13}{32})^2  dx=\frac{2197}{98304}=0.022349>V_2,\]
which is a contradiction. So, we can assume that $a_1\leq \frac{13}{32}$. We now show that the Voronoi region of $a_1$ does not contain any point from $D$. For the sake of contradiction, assume that the Voronoi region of $a_1$ contains points from $D$. Then, the following two cases can arise:

Case~1. $\frac 23 \leq \frac 12(a_1+a_2)<\frac 56$.

Then, $a_1=E(X : X \in C \uu\set{\frac 23})=\frac{17}{48}$  and $a_2=E(X : X\in \set{\frac 56, 1})=\frac{11}{12}$, and so
$\frac 12(a_1+a_2)=\frac{61}{96}<\frac 23$, which is a contradiction.

Case~2. $\frac 56 \leq \frac 12(a_1+a_2)<1$.

Then, $a_1=E(X : X \in C \uu\set{\frac 23, \frac 56})=\frac{9}{20}$  and $a_2=1$, and so
$\frac 12(a_1+a_2)=\frac{29}{40}<\frac 56$, which is a contradiction.

By Case~1 and Case~2, we can assume that the Voronoi region of $a_1$ does not contain any point from $D$. We now show that the Voronoi region of $a_2$ does not contain any point from $C$. Suppose that the Voronoi region of $a_2$ contains points from $C$. Then, the distortion error is given by
\begin{align*} &\int_{[0, \frac 12(a_1+a_2)]}(x-a_1)^2 dx+\int_{[\frac 12(a_1+a_2), \frac 12]}(x-a_2)^2 dx+\frac 16 \sum_{x\in D}(x-a_2)^2 \\
&=\frac{1}{108} \left(27 a_1^3+27 a_1^2 a_2-27 a_1 a_2^2-27 a_2^3+108 a_2^2-117 a_2+43\right),
\end{align*}
which is minimum when $a_1=\frac{5}{24}$ and $a_2=\frac{19}{24}$, and the minimum value is $\frac{37}{1728}=0.021412>V_2$, which leads to a contradiction. So, we can assume that the Voronoi region of $a_2$ does not contain any point from $C$. Thus, we have $a_1=\frac 14$ and $a_2=\frac 56$, and the corresponding quantization error is $V_2=\frac{17}{864}=0.0196759$. This, completes the proof of the lemma.
\qed

\begin{lemma1}
Let $\ga$ be an optimal set of three-means. Then, $\ga=\set{0.191074, 0.573223, \frac {11}{12}}$ with quantization error $V_3=0.0106152.$
\end{lemma1}
\nd\tbf{Proof:}
Let us consider the set of three-points $\gb:=\set {0.191074, 0.573223, \frac {11}{12}}$. Since $0.382149=\frac{1}{2} (0.191074\, +0.573223)<\frac{1}{2}<\frac{2}{3}<\frac{1}{2} \left(0.573223\, +\frac{11}{12}\right)=0.744945<\frac{5}{6}$, the distortion error due to the set $\gb$ is given by \begin{align*}& \int \min_{b \in \gb}\|x-b\|^2 dP= \int_{[0, \, 0.382149]}(x-0.191074)^2 dx+\int_{[0.382149, \, \frac 12]}(x-0.573223)^2 dx\\
&+\frac 16 (\frac 23-0.573223)^2 +\frac 16(\frac 56-\frac {11}{12})^2+\frac 16(1-\frac {11}{12})^2=0.0106152.
\end{align*}
Since $V_3$ is the quantization error for three-means, we have $V_3\leq 0.0106152$. Let $\ga:=\set{a_1, a_2, a_3}$ be an optimal set of three-means with $a_1<a_2<a_3$. Since the optimal points are the centroids of their own Voronoi regions, we have $0<a_1<a_2<a_3\leq 1$. If $\frac{3}{8}\leq a_1$, then
\[V_3\geq \int_{[0, \frac 3 8]}(x-\frac 38)^2 dx=\frac{9}{512}=0.0175781>V_3,\]
which leads to a contradiction. So, we can assume that $a_1<\frac 38$. If the Voronoi region of $a_2$ does not contain any point from $C$, then as the points of $D$ are equidistant from each other with equal probability, we will have either $a_2=\frac 12(\frac 23+\frac 56)=\frac 34$ and $a_3=1$, or $a_2=\frac 23$ and $a_3=\frac 12(\frac 56+1)=\frac {11}{12}$. In any case, the distortion error is
\[\int_{[0, \frac 12]}(x-\frac 14)^2 dx+\frac 16((\frac 23-\frac 34)^2+(\frac 56-\frac 34)^2)=\frac{11}{864}=0.0127315>V_3,\]
which is a contradiction. So, we can assume that the Voronoi region of $a_2$ contains points from $C$. If the Voronoi region of $a_2$ does not contain any point from $D$, we must have $a_1=\frac 18$, $a_2=\frac 38$, and $a_3=\frac 56$. Then, the distortion error is
\[\int_{[0, \frac 14]}(x-\frac 18)^2 dx+\int_{[\frac 14, \frac 12]}(x-\frac 38)^2 dx+\frac 16((\frac 23-\frac 56)^2+(1-\frac 56)^2)=\frac{41}{3456}=0.0118634>V_3,\]
which leads to a contradiction. Therefore, we can assume that the Voronoi region of $a_2$ contains points from $C$ as well as from $D$. We now show that the Voronoi region of $a_2$ contains only the point $\frac 23$ from $D$. On the contrary, assume that the Voronoi region of $a_2$ contains the points $\frac 23$ and $\frac 56$ from $D$. Then, we must have $a_3=1$, and so the distortion error is
\begin{align*} &\int_{[0, \frac{a_1+a_2}{2}]} (x-a_1)^2 \, dx+\int_{[\frac{a_1+a_2}{2}, \frac{1}{2}]} (x-a_2)^2 \, dx+\frac{1}{6}\Big( (\frac{2}{3}-a_2)^2+(\frac{5}{6}-a_2)^2\Big)\\
&=\frac{1}{108} \Big(27 a_1^3+27 a_1^2 a_2-27 a_1 a_2^2-27 a_2^3+90 a_2^2-81 a_2+25\Big),
\end{align*}
 which is minimum when $a_1=\frac 14$ and $a_2=\frac 34$, and the minimum value is $\frac{11}{864}=0.0127315>V_3$, which is a contradiction. Therefore, the Voronoi region of $a_2$ contains only the point $\frac 23$ from $D$. This implies $a_3=\frac 12(\frac 56+1)=\frac {11}{12}$, and then the distortion error is
\begin{align*} & \int_{[0, \frac{a_1+a_2}{2}]} (x-a_1)^2 \, dx+\int_{[\frac{a_1+a_2}{2}, \frac{1}{2}]} (x-a_2)^2 \, dx+\frac{1}{6} (\frac{2}{3}-a_2)^2+\frac{1}{6} \left((1-\frac{11}{12})^2+(\frac{5}{6}-\frac{11}{12})^2\right)\\
&=\frac{1}{144} \left(36 a_1^3+36 a_1^2 a_2-36 a_1 a_2^2-36 a_2^3+96 a_2^2-68 a_2+17\right),
\end{align*}
which is minimum when $a_1=0.191074$ and $a_2=0.573223$, and the corresponding distortion error is $V_3=0.0106152$. Moreover, we have seen $a_3=\frac{11}{12}$. Thus, the proof of the lemma is complete.
\qed

\begin{lemma1}
Let $\ga$ be an optimal set of four-means. Then, $\ga=\set{\frac 14, \frac 38, \frac 34, 1}$, or $\ga=\set{\frac 14, \frac 38, \frac 23, \frac {11}{12}}$, and the quantization error is $V_4=\frac{17}{3456}=0.00491898$.
\end{lemma1}
\nd\tbf{Proof:}
Let us consider the set of four-points $\gb:=\set{\frac 14, \frac 38, \frac 34, 1}$. Then, the distortion error due to the set $\gb$ is
 \begin{align*}& \int \min_{b \in \gb}\|x-b\|^2 dP=\int_{[0, \frac 14]}(x-\frac 18)^2 dx+\int_{[\frac 14, \frac 12]}(x-\frac 38)^2 dx+\frac 16\Big((\frac 23-\frac 34)^2+(\frac 56-\frac 34)^2\Big)=\frac{17}{3456}.
\end{align*}
Since $V_4$ is the quantization error for four-means, we have $V_4\leq\frac{17}{3456}=0.00491898$. Let $\ga:=\set{a_1<a_2<a_3<a_4}$ be an optimal set of four-means. Since the optimal points are the centroids of their own Voronoi regions, we have $0<a_1<\cdots<a_4\leq 1$.
If the Voronoi region of $a_2$ does not contain points from $C$, then
\[V_4\geq \int_{[0, \frac 12]}(x-\frac 14)^2 dx=\frac 1{96}=0.0104167>V_4,\]
which gives a contradiction, and so, we can assume that the Voronoi region of $a_2$ contains points from $C$. If the Voronoi region of $a_2$ contains points from $D$, then it can contain only the point $\frac 23$ from $D$, and in that case $a_3=\frac 56$ and $a_4=1$, which leads to the distortion error as
\begin{align*}
&\int_{[0, \frac{a_1+a_2}{2}]} (x-a_1)^2 \, dx+\int_{[\frac{a_1+a_2}{2}, \frac{1}{2}]} (x-a_2)^2 \, dx+\frac{1}{6} (\frac{2}{3}-a_2)^2\\
&=\frac{1}{216} \left(54 a_1^3+54 a_1^2 a_2-54 a_1 a_2^2-54 a_2^3+144 a_2^2-102 a_2+25\right),
\end{align*}
which is minimum when $a_1=0.191074$ and $a_2=0.573223$, and then, the minimum value is $0.00830043>V_4$, which is a contradiction. So, the Voronoi region of $a_2$ does not contain any point from $D$.
If the Voronoi region of $a_3$ does not contain any point from $D$, then $a_4=\frac 56$ yielding
\[V_4\geq \frac 16\Big((\frac 23-\frac 56)^2+(1-\frac 56)^2\Big)=\frac{1}{108}=0.00925926>V_4,\]
which leads to a contradiction. So, the Voronoi region of $a_3$ contains at least one point from $D$. Suppose that the Voronoi region of $a_3$ contains points from $C$ as well. Then, the following two cases can arise:

Case~1. $\frac 23 \in M(a_3|\ga)$.

Then, $a_4=\frac{11}{12}$, and the distortion error is
\begin{align*}
&\int_{[0, \frac{a_1+a_2}{2}]} (x-a_1)^2 \, dx+\int_{[\frac{a_1+a_2}{2}, \frac{a_2+a_3}{2}]} (x-a_2)^2 \, dx+\int_{[\frac{a_2+a_3}{2}, \frac{1}{2}]} (x-a_3)^2 \, dx+\frac{1}{6} (\frac{2}{3}-a_3)^2\\
&\qquad \qquad \qquad +\frac{1}{6} \Big((1-\frac{11}{12})^2+(\frac{5}{6}-\frac{11}{12})^2\Big)\\
&=\frac{1}{144} \Big(36 a_1^3+36 a_1^2 a_2-36 a_1 a_2^2+4 (9 a_2^2-17) a_3+(96-36 a_2) a_3^2-36 a_3^3+17\Big)
\end{align*}
which is minimum if $a_1= 0.118238$, $a_2= 0.354715$, and $a_3=0.645285$, and the minimum value is $0.00506623>V_4$, which is a contradiction.

Case~2. $\set{\frac 23, \frac 56} \sci M(a_3|\ga)$.

Then, $a_4=1$, and the corresponding distortion error is
\begin{align*}
&\int_{[0, \frac{a_1+a_2}{2}]} (x-a_1)^2 \, dx+\int_{[\frac{a_1+a_2}{2}, \frac{a_2+a_3}{2}]} (x-a_2)^2 \, dx+\int_{[\frac{a_2+a_3}{2}, \frac{1}{2}]} (x-a_3)^2 \, dx\\
&+\frac{1}{6} \Big((\frac{2}{3}-a_3)^2+(\frac{5}{6}-a_3)^2\Big)\\
&=\frac{1}{108} \Big(27 a_1^3+27 a_1^2 a_2-27 a_1 a_2^2+27 (a_2^2-3) a_3+(90-27 a_2) a_3^2-27 a_3^3+25\Big),
\end{align*}
which is minimum if $a_1=0.0990219, \, a_2=0.297066$, and $a_3=0.702934$, and the minimum value is $0.00680992>V_4$, which gives a contradiction.

By Case~1 and Case~2, we can assume that the Voronoi region of $a_3$ does not contain any point from $C$. Thus, we have $(a_1=\frac 14$, $a_2=\frac 38$, $a_3=\frac 34$, and $a_4=1$), or $(a_1=\frac 14$, $a_2=\frac 38$, $a_3=\frac 23$, and $a_4=\frac {11}{12})$, and the corresponding quantization error is $V_4=\frac{17}{3456}=0.00491898$.
\qed

\begin{lemma1} \label{lemma5111}
Let $\ga$ be an optimal set of five-means. Then, $\ga=\set{\frac 18, \frac 38, \frac 23, \frac 56, 1}$, and the corresponding quantization error is $V_5=\frac{1}{384}=0.00260417$.
\end{lemma1}
\nd\tbf{Proof:}  Consider the set of five points $\gb:=\set{\frac 14, \frac 38, \frac 23, \frac 56, 1}$. The distortion error due to the set $\gb$ is given by
\begin{align*}& \int \min_{b \in \gb}\|x-b\|^2 dP=\int_{[0, \frac 14]}(x-\frac 18)^2 dx+\int_{[\frac 14, \frac 12]}(x-\frac 38)^2 dx=\frac{1}{384}=0.00260417.
\end{align*}
Since $V_5$ is the quantization error for five-means, we have $V_5\leq 0.00260417$. Let $\ga:=\set{a_1<a_2<a_3<a_4<a_5}$ be an optimal set of five-means. Since the optimal points are the centroids of their own Voronoi regions, we have $0<a_1<\cdots<a_5\leq 1$. If the Voronoi region of $a_3$ does not contain any point from $D$, then we must have $a_1=\frac{1}{12}$, $a_2=\frac 14$, $a_3=\frac 5{12}$, $a_4=\frac 34$, and $a_4=1$, or $a_1=\frac{1}{12}$, $a_2=\frac 14$, $a_3=\frac 5{12}$, $a_4=\frac 23$, and $a_4=\frac{11}{12}$ yielding the distortion error
\[3 \int_{[0, \frac 16]}(x-\frac 1{12})^2 dx +\frac{1}{6} \Big((\frac{2}{3}-\frac{3}{4})^2+(\frac{5}{6}-\frac{3}{4})^2\Big)=\frac{1}{288}=0.00347222>V_5,\]
which is a contradiction. So, we can assume that the Voronoi region of $a_3$ contains a point from $D$. In that case, we must have $a_4=\frac 56$ and $a_5=1$. Suppose that the Voronoi region of $a_3$ contains points from $C$ as well. Then, the distortion error is
\begin{align*}
&\int_{[0, \frac{a_1+a_2}{2}]} (x-a_1)^2 \, dx+\int_{[\frac{a_1+a_2}{2}, \frac{a_2+a_3}{2}]} (x-a_2)^2 \, dx+\int_{[\frac{a_2+a_3}{2}, \frac{1}{2}]} (x-a_3)^2 \, dx+\frac{1}{6} \Big(\frac{2}{3}-a_3\Big)^2\\
&=\frac{1}{216} \left(54 a_1^3+54 a_1^2 a_2-54 a_1 a_2^2+6 \left(9 a_2^2-17\right) a_3-18 (3 a_2-8) a_3^2-54 a_3^3+25\right),
\end{align*}
which is minimum if $a_1=0.118238$, $a_2=0.354715$, and $a_3=0.645285$, and the minimum value is $0.00275142>V_5$, which is a contradiction. So, the Voronoi region of $a_3$ does not contain any point from $C$ yielding $a_1=\frac 18$, $a_2=\frac 38$, $a_3=\frac 23$, $a_4=\frac 56$ and $a_5=1$, and the corresponding quantization error is $V_5=\frac{1}{384}=0.00260417$. Thus, the proof of the lemma is complete.
\qed

\begin{theo1} \label{th61}
Let $n\in \D N$ and $n\geq 5$, and let $\ga_n$ be an optimal set of $n$-means for $P$ and $\ga_n(P_1)$ be the optimal set of $n$-means with respect to $P_1$. Then,
\[\ga_n(P)=\ga_{n-3}(P_1)\uu D, \te{ and } V_n(P)=\frac 12 V_{n-3}(P_1).\]
\end{theo1}

\nd\tbf{Proof:}  If $n=5$, by Lemma~\ref{lemma5111}, we have $\ga_5(P)=\set{\frac 18, \frac 38, \frac 23, \frac 56, 1}$ and $V_5(P)=\frac 1{384}$, which by Corollary~\ref{cor1} yields that $\ga_5(P)=\ga_2(P_1)\uu D$ and $V_5(P)=\frac 12 V_2(P_1)$, i.e., the theorem is true for $n=5$. Proceeding in the similar way, as Lemma~\ref{lemma5111}, we can show that the theorem is true for $n=6$ and $n=7$. We now show that the theorem is true for all $n\geq 8$. Consider the set of eight points $\gb:=\set{\frac 1{20}, \frac{3}{20}, \frac 1 4, \frac 7{20}, \frac{9}{20}, \frac 23, \frac 56, 1}$. The distortion error due to set $\gb$ is given by
\[\int \min_{b \in \gb}\|x-b\|^2 dP=5 \int_{[0, \frac 1{10}]}(x-\frac1 {20})^2 dx=\frac{1}{2400}=0.000416667.\]
Since $V_n$ is the $n$th quantization error for $n$-means for $n\geq 8$, we have $V_n\leq V_8\leq 0.000416667$. Let $\ga_n:=\set{a_1<a_2<\cdots <a_n}$ be an optimal set of $n$-means for $n\geq 8$, where $0<a_1<\cdots<a_n\leq 1$. To prove the first part of the theorem, it is enough to show that $M(a_{n-2}|\ga_n)$ does not contain any point from $C$, and $M(a_{n-3}|\ga_n)$
 does not contain any point from $D$. If $M(a_{n-2}|\ga_n)$ does not contain any point from $D$, then
 \[V_n\geq \frac 16\Big((\frac 23-\frac 34)^2+(\frac 56-\frac 34)^2\Big)=\frac{1}{432}=0.00231481>V_n,\]
which leads to a contradiction. So, $M(a_{n-2}|\ga_n)$ contains a point, in fact the point $\frac 23$, from $D$. If $M(a_{n-2}|\ga_n)$ does not contain points from $C$, then $a_{n-2}=\frac 23$. Suppose that $M(a_{n-2}|\ga_n)$ contains points from $C$. Then, $\frac 23\leq \frac 12(a_{n-2}+a_{n-1})$ implies $a_{n-2}\geq \frac 43-a_{n-1}=\frac 43-\frac 56=\frac 12$. The following three cases can arise:

Case~1. $\frac 12\leq a_{n-2}\leq \frac 7{12}$.

Then, $V_n\geq \frac 16 (\frac 23-\frac 7{12})^2=\frac{1}{864}=0.00115741>V_n,$ which is a contradiction.

Case~2. $\frac 7{12}\leq a_{n-2}\leq \frac 58$.

Then, $\frac 12(a_{n-3}+a_{n-2})<\frac 12$ implying $a_{n-3}<1-a_{n-2}\leq 1-\frac 7{12}=\frac 5{12}$, and so
\[V_n\geq \int_{[\frac{5}{12}, \frac{1}{2}]} \Big(x-\frac{5}{12}\Big)^2 \, dx+\frac{1}{6} \Big(\frac{2}{3}-\frac{5}{8}\Big)^2=\frac{5}{10368}=0.000482253>V_n,\]
which leads to a contradiction.

Case~3. $\frac 58 \leq a_{n-2}$.

Then, $\frac 12(a_{n-3}+a_{n-2})<\frac 12$ implying $a_{n-3}<1-a_{n-2}\leq 1-\frac 5{8}=\frac 38$, and so
\[V_n\geq \int_{[\frac{3}{8}, \frac{1}{2}]} \Big(x-\frac{3}{8}\Big)^2 \, dx=\frac{1}{1536}=0.000651042>V_n,\]
which gives contradiction.

Thus, in each case we arrive at a contradiction yielding the fact that $M(a_{n-2}|\ga_n)$ does not contain any point from $C$. If $M(a_{n-3}|\ga)$ contains any point from $D$, say $\frac 23$, then we will have
\[M(a_{n-2}|\ga)\uu M(a_{n-1}|\ga)\uu M(a_n|\ga)=\set{\frac 56, 1},\]
which by Proposition~\ref{prop0} implies that either $a_{n-2}=a_{n-1}=\frac 56$, and $a_n=1$, or $a_{n-2}=\frac 56$, and $a_{n-1}=a_n=1$, which contradicts the fact that $0<a_1<\cdots<a_{n-2}<a_{n-1}<a_n\leq 1$. Thus, $M(a_{n-3}|\ga)$ does not contain any point from $D$. Hence, $\ga_n(P)=\ga_{n-3}(P_1)\uu D$, and so,
 \[V_n(P)=\int_{C}\min_{a\in\ga_{n-3}(P_1)}(x-a)^2 dx+\frac 16 \sum_{x\in D}\min_{a\in D}(x-a)^2=\frac 12\int_{C}\min_{a\in\ga_{n-3}(P_1)}(x-a)^2 2dx\]
 implying $V_n(P)=\frac 12 V_{n-3}(P_1)$. Thus, the proof of the theorem is complete.
\qed
\begin{propo} \label{prop612}
Let $P$ be the mixed distribution as defined before. Then,
\[\lim_{n\to \infty} n^2 V_n(P)=\frac 1{96}.\]
\end{propo}

\nd\tbf{Proof:}  By Corollary~\ref{cor1} and Theorem~\ref{th61}, we have
\[\lim_{n\to \infty} n^2 V_n(P)=\frac 12 \lim_{n\to \infty} n^2 V_{n-3} (P_1)=\frac12 \lim_{n\to \infty} \frac {n^2} {48(n-3)^2}=\frac 1{96},\]
and thus, the proposition is yielded.
\qed
\begin{rem}\label{rem001}
By Proposition~\ref{prop612}, it follows that $\mathop{\lim}\limits_{n\to\infty} n^2 V_n(P)=\frac 1{96}$, i.e., one-dimensional quantization coefficient for the mixed distribution $P$ is finite and positive implying the fact that the quantization dimension of the mixed distribution $P$ exists, and equals one, which is the dimension of the underlying space. It is known that for a probability measure $P$ on $\D R^d$ with non-vanishing absolutely continuous part $\mathop{\lim}\limits_{n\to\infty} n^{\frac 2d} V_n(P)$ is finite and strictly positive, i.e., the quantization dimension of $P$ exists, and equals the dimension $d$ of the underlying space (see \cite{BW}). Thus, for the mixed distribution $P$ considered in this section, we see that $D(P)=D(P_1)=1$.
\end{rem}

\section{A rule to determine optimal quantizers} \label{sec3}
Let $0<p<1$ be fixed. Let $P$ be a mixed distribution given by $P=pP_1+(1-p)P_2$  with the support of $P_1$ equals $C$ and the support of $P_2$ equals $D$, such that $P_1$ is continuous on $C$, and $P_2$ is discrete on $D$, and $D\sci C$. It is well-known that the optimal set of one-mean consists of the expected value and the corresponding quantization error is the variance $V$ of the $P$-distributed random variable $X$. Assume that $P_1$ is absolutely continuous on $C:=[0, 1]$, and $P_2$ is discrete on $D:=\set{0, 1}$. Then, in the following note we give a rule how to obtain the optimal sets of $n$-means for the mixed distribution $P$ for any $n\geq 2$.
\begin{note}\label{note2222}
Let $\ga_n:=\set{a_1, a_2, \cdots, a_n}$ be an optimal set of $n$-means for $P$ such that $0\leq a_1<a_2<\cdots<a_n\leq 1$. Write
\begin{align}\label{eq234}
M(a_i|\ga_n):=\left\{\begin{array}{cc}
\left[0, \frac{a_1+a_2}{2}\right] & \te{ if } i=1, \\
\left[\frac{a_{i-1}+a_i}{2}, \frac{a_i+a_{i+1}}{2}\right] & \te{ if } 2 \leq i \leq n-1, \\
 \left[\frac{a_{n-1}+a_n}{2}, 1\right] & \te{ if } i=n,
\end{array}
\right.
\end{align}
where $M(a_i|\ga_n)$ represent the Voronoi regions of $a_i$ for all $1\leq i\leq n$ with respect to the set $\ga_n$.
Since the optimal points are the centroids of their own Voronoi regions, we have $a_i=E(X : X \in M(a_i|\ga_n))$ for all $1\leq i\leq n$. Solving the $n$ equations one can obtain the optimal sets of $n$-means for the mixed distribution $P$. Once, an optimal set of $n$-means is known, the corresponding quantization error can easily be determined.
\end{note}
Let us now give the following proposition.

\begin{prop} \label{prop1111}
Let $\ga_n$ be an optimal set of $n$-means and $V_n$ is the corresponding quantization error for $n\geq 2$ for the mixed distribution $P:=\frac 12 P_1+\frac 12 P_2$ such that $P_1$ is uniformly distributed on $C:=[0, 1]$ with probability density function $g$ given by
\[ g(x)=\left\{\begin{array}{ccc}
1 & \te{ if }x \in C,\\
 0  & \te{ otherwise},
\end{array}\right.
\] and $P_2$ is discrete on $D:=\set{1}$ with  mass function $h$ given by  $h(1)=1$. Then, for $n\geq 2$,
\[\ga_n:=\left\{\frac{(2 i-1) \left(-\sqrt{n^2-n+1}+2 n-1\right)}{2 (n-1) n} : 1\leq i\leq n \right\}\]
and $V_n=\frac{4 n^2-4 \left(\sqrt{n^2-n+1}+1\right) n+2 \sqrt{n^2-n+1}+7}{12 \left(\sqrt{n^2-n+1}+2 n-1\right)^2}.$
\end{prop}
\nd\tbf{Proof:}  As mentioned in Note~\ref{note2222}, solving the $n$ equations $a_i=E(X : X \in M(a_i|\ga))$, we obtain
\[a_i=\frac{(2 i-1) \left(-\sqrt{n^2-n+1}+2 n-1\right)}{2 (n-1) n},\]
for all $1\leq i\leq n$, and hence, the corresponding quantization error is given by
\begin{align*} V_n=\int_0^{\frac 12(a_1+a_2)}(x-a_1)^2 dx+ \sum_{i=2}^{n-1}\int_{\frac 12(a_{i-1}+a_{i})}^{\frac 12(a_{i}+a_{i+1})}(x-a_{i})^2 dx+\int_{\frac 12(a_{n-1}+a_{n})}^1(x-a_n)^2 dx+\frac 12(a_n-1)^2,
\end{align*}
which upon simplification yields $V_n=\frac{4 n^2-4 \left(\sqrt{n^2-n+1}+1\right) n+2 \sqrt{n^2-n+1}+7}{12 \left(\sqrt{n^2-n+1}+2 n-1\right)^2}$. Thus, the proof of the proposition is complete.
\qed

\begin{rem} \label{rem4}
Let $P_1$ be absolutely continuous on $C:=[0, 1]$ and $P_2$ be discrete on $D$ with $D\sci C$. Then, if $D:=\set{0, 1}$, the system of equations in \eqref{eq234} has a unique solution implying that there exists a unique optimal set of $n$-means for the mixed distribution $P:=pP_1+(1-p)P_2$ for each $n\in \D N$. If $D\ii \te{Int}(C)$ is nonempty, where $\te{Int}(C)$ represents the interior of $C$, then as it is seen in Proposition~\ref{prop1112}, the optimal sets of $n$-means for the mixed distribution $P$ for all $n\in \D N$ is not necessarily unique.
\end{rem}

\begin{prop} \label{prop1112} Let $P:=\frac 12 P_1+\frac 12 P_2$, where $P_1$ is uniformly distributed on $C:=[0, 1]$ and $P_2$ is discrete on $D:=\set{\frac 12}$. Then, $P$ has two different optimal sets of two-means.
\end{prop}

\nd\tbf{Proof:}
Let $\ga:=\set{a_1, a_2}$ be an optimal set of two means for $P$ with $0<a_1<a_2<1$. Then, $P$-almost surely, we have $C=M(a_1|\ga)\uu M(a_2|\ga)$ implying that either $\frac 12 \in M(a_1|\ga)$, or $\frac 12\in M(a_2|\ga)$. First, assume that $\frac 12 \in M(a_1|\ga)$, i.e., $0<a_1<\frac 12\leq \frac 12(a_1+a_2)$. Then,
\begin{align*}  & a_1=E(X : X \in [0, \frac 12(a_1+a_2)])=\frac{\int_0^{\frac{a+b}{2}} x \, dx+\frac{1}{2}}{\int_0^{\frac{a+b}{2}} 1 \, dx+1}=\frac{a^2+2 a b+b^2+4}{4 (a+b+2)}, \te{ and } \\
& a_2=E(X : X \in [\frac 12(a_1+a_2), 1])=\frac{\int_{\frac{a+b}{2}}^1 x \, dx}{\int_{\frac{a+b}{2}}^1 1 \, dx}=\frac{1}{4} (a+b+2).
\end{align*}
Solving the above two equations, we have
$a_1=\frac{1}{4} (-5+3 \sqrt{5})$ and $a_2=\frac{1}{4} (1+\sqrt{5})$, and the corresponding quantization error is given by
 \begin{align*}
& V_2(P)=\int\min_{a\in\ga}\|x-a\|^2 dP=\frac 12 \int\min_{a\in\ga}(x-a)^2 dP_1+\frac 12\int\min_{a\in\ga}(x-a)^2 dP_2\\
&=\frac{1}{2} \int_0^{\frac{a_1+a_2}{2}} (x-a_1)^2 \, dx+\frac{1}{2} \int_{\frac{a_1+a_2}{2}}^1 (x-a_2)^2 \, dx+\frac{1}{2} \Big(\frac{1}{2}-a_1\Big)^2=0.0191242.
 \end{align*}
Next, assume that $\frac 12 \in M(a_2|\ga)$, i.e., $\frac 12(a_1+a_2)\leq \frac 12<a_2<1$. Then,
\begin{align*}  & a_1=E(X : X \in [0, \frac 12(a_1+a_2)])=\frac{\int_0^{\frac{a+b}{2}} x \, dx}{\int_0^{\frac{a+b}{2}} 1 \, dx}=\frac{a+b}{4}, \te{ and } \\
& a_2=E(X : X \in [\frac 12(a_1+a_2), 1])=\frac{\int_{\frac{a+b}{2}}^1 1 x \, dx+\frac{1}{2}}{\int_{\frac{a+b}{2}}^1 1 \, dx+1}=\frac{a^2+2 a b+b^2-8}{4 (a+b-4)}.
\end{align*}
Solving the above two equations, we have
$a_1=\frac{1}{4} (3-\sqrt{5})$ and $a_2=\frac{3}{4}(3-\sqrt{5})$, and as before, the corresponding quantization error is give by
\begin{align*}
& V_2(P)=\frac{1}{2} \int_0^{\frac{a_1+a_2}{2}} (x-a_1)^2 \, dx+\frac{1}{2} \int_{\frac{a_1+a_2}{2}}^1 (x-a_2)^2 \, dx+\frac{1}{2} \Big(\frac{1}{2}-a_2\Big)^2=0.0191242.
 \end{align*}
 Thus, we see that there are two different optimal sets of two-means with same quantization error, which is the proposition.
 \qed

\section{Quantization with $P_1$ a Cantor distribution and $P_2$ discrete}\label{sec5}

In this section, we consider a mixed distribution $P:=\frac 12 P_1+\frac 12 P_2$, where $P_1$ is a Cantor distribution given by $P_1=\frac 12 P_1\circ S_1^{-1}+\frac 12 P_1\circ S_2^{-1}$, where $S_1(x)=\frac 13 x$ and $S_2(x)=\frac 13 x+\frac 13$ for all $x\in \D R$, and $P_2$ is a discrete distribution on $D:=\set{\frac 23, \frac 56, 1}$ with density function $h$ given by $h(x)=\frac 13$ for all $x\in D$.
By a \tit{word}, or a \tit {string} of length $k$ over the alphabet $\set{1, 2}$, it is meant $\gs:=\gs_1\gs_2\cdots \gs_k$, where $\gs_j\in \set{1, 2}$ for $1\leq j\leq k$. A word of length zero is called the empty word and is denoted by $\es$. Length of a word $\gs$ is denoted by $|\gs|$. The set of all words over the alphabet $\set{1, 2}$ including the empty word $\es$ is denoted by $\set{1, 2}^\ast$. For two words $\gs:=\gs_1\gs_2\cdots\gs_{|\gs|}$ and $\gt:=\gt_1\gt_2\cdots \gt_{|\gt|}$, by $\gs\gt$, it is meant the concatenation of the words $\gs$ and $\gt$. If $\gs=\gs_1\gs_2\cdots\gs_k$, we write $S_\gs:=S_{\gs_1}\circ  S_{\gs_2}\circ\cdots\circ S_{\gs_k}$, and $J_\gs=S_\gs(J)$, where $J=J_\es:=[0, \frac 12]$. $S_1$ and $S_2$ generate the Cantor set $C:=\bigcap_{k\in \mathbb N} \bigcup_{\gs \in \{1, 2\}^k} J_\gs$. $C$ is the support of the probability distribution $P_1$. Notice that the support of the Mixed distribution $P$ is $C\uu D$. For any $\gs \in \set{1, 2}^k$, $k\geq 1$, the intervals $J_{\gs1}$ and $J_{\gs2}$ into which $J_\gs$ is split up at the $(k+1)$th level are called the \tit{children} of $J_\gs$.

The following lemma is well-known and appears in many places, for example, see \cite{GL2, R1}.

\begin{lemma} \label{lemma51}
Let $f : \mathbb R \to \mathbb R^+$ be Borel measurable and $k\in \mathbb N$. Then
\[\int f dP_1=\sum_{\sigma \in \{1, 2\}^k} \frac 1 {2^k} \int f \circ S_\sigma dP_1.\]
\end{lemma}

\begin{lemma} \label{lemma52}
Let $X_1$ be a $P_1$-distributed random variable. Then, its expectation and the variance are respectively given by $E(X_1)=\frac  14$ and $V(X_1)=\frac {1}{32},$ and for any $x_0 \in \mathbb R$,
$\int (x-x_0)^2 dP_1(x) =V (X_1) +(x_0-\frac 14)^2.$
\end{lemma}
\nd\tbf{Proof:}  Using Lemma~\ref{lemma51}, we have
\[E(X_1)=\int x \, dP_1=\frac 1 2\int \frac 1 3 x \, dP_1 +\frac 12  \int (\frac 13 x +\frac 13 ) \, dP_1= \frac 1{6} \,E(X_1) +\frac 1{6} \, E(X_1) +\frac 1{6}\] implying $E(X_1)=\frac 1 4.$ Again,
\begin{align*}
&E(X_1^2)=\int x^2 \, dP_1=\frac 1 2  \int \frac {1} {9}\, x^2 \, dP_1 +\frac 12   \int \Big(\frac 1 3 \, x +\frac 13 \Big)^2 \, dP_1= \frac {1}{9}\, E(X_1^2)  +\frac {1}{9} \, E(X_1) +\frac {1} {18},
\end{align*}
which yields $E(X_1^2)=\frac {3}{32}$, and hence
$V(X_1)=E(X_1-E(X_1))^2=E(X_1^2)-\left(E(X_1)\right)^2 =\frac {3}{32}-(\frac 14 )^2=\frac {1}{32}$. Then, following the standard theory of probability, we have
$ \int(x-x_0)^2 \, dP_1 =V(X_1)+(x_0-E(X_1))^2,$ and thus the lemma is yielded.
\qed

\begin{defi}  \label{defi5111} For $n\in \D N$ with $n\geq 2$, let $\ell(n)$ be the unique natural number with $2^{\ell(n)} \leq n<2^{\ell(n)+1}$. For $I\sci \set{1, 2}^{\ell(n)}$ with card$(I)=n-2^{\ell(n)}$ let $\gb_n(I)$ be the set consisting of all midpoints $a(\gs)$ of intervals $J_\gs$ with $\gs \in \set{1,2}^{\ell(n)} \setminus I$ and all midpoints $a(\gs 1)$, $a(\gs 2)$ of the children of $J_\gs$ with $\gs \in I$, i.e.,
\[\gb_n(I)=\set{a(\gs) : \gs \in \set{1,2}^{\ell(n)} \setminus I} \uu \set{a(\gs 1) : \gs \in I} \uu \set {a(\gs 2) : \gs \in I}.\]
\end{defi}
The following proposition follows due to  \cite[Definition~3.5 and Proposition~3.7]{GL2}.

\begin{prop}  \label{prop5111}  Let $\gb_n(I)$ be the set for $n\geq 2$ given by Definition~\ref{defi5111}. Then, $\gb_n(I)$ forms an optimal set of $n$-means for $P_1$, and the corresponding quantization error is given by
\[V_n(P_1)=\int\min_{a\in \gb_n(I)}\|x-a\|^2 \, dP_1=\frac 1{18^{\ell(n)}}\cdot\frac 1{32} \Big(2^{\ell(n)+1}-n+\frac 19\left(n-2^{\ell(n)}\right)\Big).\]
\end{prop}

\begin{lemma}\label{lemma53}
Let $E(X)$ and $V:=V(X)$ represent the expected value and the variance of a random variable $X$ with distribution $P$. Then, $E(X)=\frac{13}{24}$ and $V=\frac{95}{864}=0.109954$.
\end{lemma}
\nd\tbf{Proof:}  In this proof we use the results from Lemma~\ref{lemma51}. We have
\begin{align*} &E(X)=\int x dP=\frac 12 \int x dP_1+\frac 12 \int x dP_2=\frac 12 \int x dP_1+\frac 12 \sum_{x\in D} x h(x)=\frac{13}{24}, \te{ and } \\ &E(X^2)=\int x^2 dP=\frac 12 \int x^2 dP_1+\frac 12 \sum_{x\in D} x^2 h(x)=\frac{697}{1728},
\end{align*}
implying $V:=V(X)=E(X^2)-(E(X))^2 =\frac{697}{1728}-\left(\frac{13}{24}\right)^2=\frac{95}{864}$. Thus, the lemma is yielded.
\qed

\begin{note}
Since $E\|X-a\|^2 =\int(x-a)^2 dP=V(X)+(a-E(X))^2=V+(a-\frac{13}{24})^2$, it follows that the optimal set of one-mean for the mixed distribution $P$ consists of the expected value $\frac {13}{24}$, and the corresponding quantization error is the variance $V$ of the random variable $X$. For any $\gs \in \set{1, 2}^\ast$, by $a(\gs)$, it is meant $a(\gs):=E(X_1 : X_1\in J_\gs)$, where $X_1$ is a $P_1$ distributed random variable, i.e., $a(\gs)=S_\gs(\frac 14)$. Notice that for any $\gs \in \set{1, 2}^\ast$, and for any $x_0 \in \mathbb R$, we have
\begin{equation} \label{eq234} \int_{J_\gs} (x-x_0)^2 \, dP_1 =p_\gs\Big(s_\gs^2 V  +(S_\gs(\frac 14)-x_0)^2\Big),\end{equation}
where $p_\gs=\frac 1 {2^{|\gs|}}$, and $s_\gs=\frac 1 {3^{|\gs|}}$.
\end{note}

\subsection{Optimal sets of $n$-means and $n$th quantization error} In this subsection, we determine the optimal sets of $n$-means and the $n$th quantization errors for all $n\geq 2$ for the mixed distribution $P$. To determine the distortion error, we will frequently use the equation~\eqref{eq234}.
\begin{lemma1}
Let $\ga$ be an optimal set of two-means. Then, $\ga=\set{\frac 14, \frac 56}$ with quantization error $V_2=\frac{43}{1728}=0.0248843.$
\end{lemma1}
\nd\tbf{Proof:}  Consider the set of two-points $\gb$ given by $\gb:=\set{\frac 14, \frac 56}$. Then, the distortion error is
\[\int \min_{b \in \gb}\|x-b\|^2 dP= \frac 12 \int_{C}(x-\frac 14)^2 dP_1+\frac 16 \sum_{x\in D} (x-\frac 56)^2=\frac{43}{1728}=0.0248843.\]
Since $V_2$ is the quantization error for two-means, we have $V_2\leq 0.0248843$. Let $\ga:=\set{a_1, a_2}$ be an optimal set of two-means with $a_1<a_2$. Since the optimal points are the centroids of their own Voronoi regions, we have $0<a_1<a_1<a_2\leq 1$. If $a_1\geq \frac {29}{72}>S_{21}(\frac 12)$, then
\[V_2\geq \frac 12\int_{J_1\uu J_{21}}(x-\frac {29}{72})^2 dP_1=\frac{1105}{41472}=0.0266445>V_2,\]
which leads to a contradiction. We now show that the Voronoi region of $a_1$ does not contain any point from $D$. Notice that the Voronoi region of $a_1$ can not contain all the points from $D$ as by Proposition~\ref{prop0}, $P(M(a_2|\ga))>0$. First, assume that the Voronoi region of $a_1$ contains both $\frac 23$ and $\frac 56$. Then,
\[a_1=E(X : X\in C\uu \set{\frac 23, \frac 56})=\frac{\frac 12 \frac 1 4+\frac 16 \frac 23+\frac 16 \frac 56}{\frac 12+\frac 16+\frac 16}=\frac{9}{20} \te{ and } a_2=1,\]
which yield $\frac 12(a_1+a_2)=\frac {29}{40}<\frac 56$, which is a contradiction, as we assumed $\set{\frac 23, \frac 56}\sci M(a_1|\ga)$. Next, assume that the Voronoi region of $a_1$ contains only the point $\frac 23$ from $D$. Then,
\[a_1=E(X : X\in C\uu \set{\frac 23})=\frac{\frac 12 \frac 1 4+\frac 16 \frac 23}{\frac 12+\frac 16}=\frac{17}{48} \te{ and } a_2=\frac 12(\frac 56+1)=\frac {11}{12},\]
which yield $\frac 12(a_1+a_2)=\frac{61}{96}<\frac 23$, which is a contradiction, as the Voronoi region of $a_1$ contains $\frac 23$. Thus, we can assume that the Voronoi region of $a_1$ does not contain any point from $D$ implying that $a_1\leq \frac 14$. Notice that if the Voronoi region of $a_1$ does not contain any point from $D$ and the Voronoi region of $a_2$ does not contain any point from $C$, then $a_1=\frac 14$ and $a_2=\frac 56$. If $a_2< \frac{21}{32}$, then
\[V_2\geq \frac 16\Big((\frac 23-\frac{21}{32})^2+(\frac 56-\frac{21}{32})^2+(1-\frac{21}{32})^2\Big)=\frac{1379}{55296}=0.0249385>V_2,\]
which gives a contradiction, and so $\frac{21}{32}\leq  a_2\leq \frac 56$. Suppose that $\frac{21}{32}\leq a_2\leq \frac{17}{24}$. Since $a_1\leq \frac 14$, $E(X_1 : X_1\in J_{1}\uu J_{21})=\frac{19}{108}<\frac 14 $, and $S_{21} (\frac 12)< \frac 12(\frac{19}{108}+\frac {21}{32})<\frac{1}{2} (\frac{1}{4}+\frac{21}{32})<S_{2212}(0)$, we have
\begin{align*}
 V_2&\geq \frac 12\Big(\int_{J_1\uu J_{21}}(x-\frac{19}{108})^2 dP_1+\int_{J_{2212}\uu J_{222}}(x-\frac {21}{32})^2 dP_1\Big)+\frac{1}{6} \Big((\frac 56-\frac{17}{24})^2+(1-\frac{17}{24})^2\Big)\\
&=\frac{1938409}{71663616}=0.0270487>V_2,
\end{align*}
which leads to a contradiction. So, we can assume that $\frac{17}{24}\leq a_2\leq \frac 56$. Suppose that $\frac{17}{24}\leq a_2\leq \frac 34$. Notice that $S_{221}(\frac 12)<\frac 12(\frac 14+\frac{17}{24})<S_{222}(0)$, and $E(X_1 : X_1\in J_{1}\uu J_{21}\uu J_{2211})=\frac{829}{4212}<\frac 14$, and so, we have
\begin{align*} &V_2\geq \frac 12 \Big(\int_{J_1\uu J_{21}\uu J_{2211}}(x-\frac{829}{4212})^2 dP_1+\int_{J_{2212}}(x-\frac 14)^2 dP_1+\int_{J_{222}}(x-\frac {17}{24})^2 dP_1\Big)\\
&+\frac{1}{6} \Big((\frac{2}{3}-\frac{17}{24})^2+(\frac{5}{6}-\frac{3}{4})^2+(1-\frac{3}{4})^2\Big)=\frac{2242573}{87340032}=0.0256763>V_2,
\end{align*}
which is a contradiction. So, we can assume that $\frac 34\leq a_2\leq \frac 56$. Then, notice that $\frac 12(a_1+a_2)<\frac 12$ implying $a_1<1-a_2\leq \frac 14$, but $\frac 12(\frac 14+\frac 34)=\frac 12$, and thus, $P$-almost surely the Voronoi region of $a_2$ does not contain any point from $C$ yielding $a_1=\frac 14$, $a_2=\frac 56$, and the corresponding quantization error is $V_2=\frac{43}{1728}=0.0248843.$
\qed

Let us now state the following three lemmas. Due to technicality we do not show the proofs in the paper.
\begin{lemma1}
Let $\ga$ be an optimal set of three-means. Then, $\ga=\set{\frac{1}{12},\frac{31}{60},\frac{11}{12}}$ with quantization error $V_3=\frac{89}{8640}=0.0103009$.
\end{lemma1}

\begin{lemma1}
Let $\ga$ be an optimal set of four-means. Then, $\ga=\set{\frac 1{12}, \frac 5{12}, \frac 34, 1}$, or $\ga=\set{\frac 1{12}, \frac 5{12}, \frac 23, \frac {11}{12}}$, and the quantization error is $V_4=\frac{7}{1728}=0.00405093$.
\end{lemma1}

\begin{lemma1} \label{lemma551}
Let $\ga$ be an optimal set of five-means. Then, $\ga=\ga_2(P_1)\uu D$, and the corresponding quantization error is $V_5=\frac{1}{576}=\frac 12 V_2(P_1)$.
\end{lemma1}

\begin{theo1} \label{th611}
Let $n\in \D N$ and $n\geq 5$, and let $\ga_n$ be an optimal set of $n$-means for $P$ and $\ga_n(P_1)$ be the optimal set of $n$-means for $P_1$. Then,
\[\ga_n(P)=\ga_{n-3}(P_1)\uu D, \te{ and } V_n(P)=\frac 12 V_{n-3}(P_1).\]
\end{theo1}

\nd\tbf{Proof:}  If $n=5$, by Lemma~\ref{lemma551}, we see that the theorem is true for $n=5$. Proceeding in the similar way, as Lemma~\ref{lemma551}, we can show that the theorem is true for $n=6$  and $n=7$. We now show that the theorem is true for all $n\geq 8$. Consider the set of eight points $\gb:=\set{a(11), a(12), a(21), a(221), a(222), \frac 23, \frac 56, 1}$. The distortion error due to set $\gb$ is given by
\[\int \min_{b \in \gb}\|x-b\|^2 dP=\frac 12 V_5(P_1)=\frac{7}{46656}=0.000150034.\]
Since $V_n$ is the $n$th quantization error for $n$-means for $n\geq 8$, we have $V_n\leq V_8\leq0.000150034$. Let $\ga_n:=\set{a_1<a_2<\cdots <a_n}$ be an optimal set of $n$-means for $n\geq 8$, where $0<a_1<\cdots<a_n\leq 1$. To prove the first part of the theorem, it is enough to show that $M(a_{n-2}|\ga_n)$ does not contain any point from $C$, and $M(a_{n-3}|\ga_n)$
 does not contain any point from $D$. If $M(a_{n-2}|\ga_n)$ does not contain any point from $D$, then
 \[V_n\geq \frac 16\Big((\frac 23-\frac 34)^2+(\frac 56-\frac 34)^2\Big)=\frac{1}{432}=0.00231481>V_n,\]
which leads to a contradiction. So, $M(a_{n-2}|\ga_n)$ contains a point, in fact the point $\frac 23$, from $D$. If $M(a_{n-2}|\ga_n)$ does not contain points from $C$, then $a_{n-2}=\frac 23$. Suppose that $M(a_{n-2}|\ga_n)$ contains points from $C$. Then, $\frac 23\leq \frac 12(a_{n-2}+a_{n-1})$ implies $a_{n-2}\geq \frac 43-a_{n-1}=\frac 43-\frac 56=\frac 12$. The following three cases can arise:

Case~1. $\frac 12\leq a_{n-2}\leq \frac 7{12}$.

Then, $V_n\geq \frac 16 (\frac 23-\frac 7{12})^2=\frac{1}{864}=0.00115741>V_n,$ which is a contradiction.

Case~2. $\frac 7{12}\leq a_{n-2}$.

Then, $\frac 12(a_{n-3}+a_{n-2})<\frac 12$ implying $a_{n-3}<1-a_{n-2}\leq 1-\frac 7{12}=\frac 5{12}$, and so
\[V_n\geq \frac 12 \int_{J_{22}} \Big(x-\frac{5}{12}\Big)^2 dP_1=\frac{1}{2304}=0.000434028>V_n,\]
which leads to a contradiction.

By Case~1 and Case~2, we can assume that $M(a_{n-2}|\ga_n)$ does not contain any point from $C$. If $M(a_{n-3}|\ga)$ contains any point from $D$, say $\frac 23$, then we will have
\[M(a_{n-2}|\ga)\uu M(a_{n-1}|\ga)\uu M(a_n|\ga)=\Big\{\frac 56, 1\Big\},\]
which by Proposition~\ref{prop0} implies that either $a_{n-2}=a_{n-1}=\frac 56$ and $a_n=1$, or $a_{n-2}=\frac 56$ and $a_{n-1}=a_n=1$, which contradicts the fact that $0<a_1<\cdots<a_{n-2}<a_{n-1}<a_n\leq 1$. Thus, $M(a_{n-3}|\ga)$ does not contain any point from $D$. Hence, $\ga_n(P)=\ga_{n-3}(P_1)\uu D$, and so,
 \[V_n(P)=\frac 12 \int_{C}\min_{a\in\ga_{n-3}(P_1)}(x-a)^2 dP_1+\frac 16 \sum_{x\in D}\min_{a\in D}(x-a)^2=\frac 12\int_{C}\min_{a\in\ga_{n-3}(P_1)}(x-a)^2 dP_1\]
 implying $V_n(P)=\frac 12 V_{n-3}(P_1)$. Thus, the proof of the theorem is complete.
\qed

\begin{rem}\label{rem00012}
 Let $\gb$ be the Hausdorff dimension of the Cantor set generated by the similarity mappings $S_1$ and $S_2$. Then, $\gb=\frac{\log 2}{\log 3}$.
By \cite[Theorem~6.6]{GL2}, it is known that the quantization dimension of $P_1$ exists and equals $\gb$, i.e., $D(P_1)=\gb$. Since
\[D(P)=\lim_{n\to \infty} \frac{2\log n}{-\log 2-\log V_{n-m} (P_1)}=\lim_{n\to \infty} \frac{2\log (n-m)}{-\log V_{n-m}(P_1)}=D(P_1)=\gb,\]
we can say that the quantization dimension of the mixed distribution exists and equals the quantization dimension of the Cantor distribution $P_1$, i.e., $D(P)=D(P_1)=\gb$.
Again, by \cite[Theorem~6.3]{GL2}, it is known that the quantization coefficient for $P_1$ does not exits. By Theorem~\ref{th611}, we have $\liminf_{n\to \infty} n^{\frac {2}{\gb}} V_n(P)=\frac 12 \liminf_{n\to \infty} n^{\frac {2}{\gb}} V_{n-3}(P_1)=\frac 12 \liminf_{n\to \infty} (n-3)^{\frac {2}{\gb}} V_{n-3}(P_1)$, and similarly,  $\limsup_{n\to \infty} n^{\frac {2}{\gb}} V_n(P)=\frac 12 \limsup_{n\to \infty} (n-3)^{\frac {2}{\gb}} V_{n-3}(P_1)$. Hence, the quantization coefficient for the mixed distribution $P$ does not exist.
\end{rem}

\section{Some remarks} \label{sec61}

Theorem~\ref{th61} and Theorem~\ref{th611} motivate us to give the following remarks.

\begin{rem} Let $0<p<1$ be fixed. Let $P$ be the mixed distribution given by $P=pP_1+(1-p)P_2$  with the support of $P_1=C$ and the support of $P_2=D$, such that $P_1$ is continuous on $C$ and $P_2$ is discrete on $D$. Let $\te{card}(D)=m$ for some positive integer $m$. Further assume that $C$ and $D$ are \tit{strongly separated}: there exists a $\gd>0$ such that $d(C, D):=\inf\set {d(x, y) : x\in C \te{ and } y\in D}>\gd$. Then, we conjecture that there exists a positive integer $N$ such that for all $n\geq N$, we have $\ga_n(P)=\ga_{n-m}(P_1)\uu D$. Notice that it is not known whether the quantization dimension $D(P_1)$ of $P_1$ exists; if $D(P_1)$ exists, then as the quantization dimension of a finite discrete distribution is zero, by Proposition~\ref{prop1000}, we can say that the quantization dimension $D(P)$ of the mixed distribution $P$ exists, and $D(P)=D(P_1)$.
\end{rem}

\begin{rem} \label{rem5}
Let $D$ be a finite discrete subset of $C:=[0, 1]$. If $P_1$ is continuous on $C$, singular or nonsingular, and $P_2$ is discrete on $D$, then for the mixed distribution $P:=pP_1+(1-p)P_2$, where $0<p<1$, the optimal sets of $n$-means and the $n$th quantization errors for all $n\geq 2$ and for all $D$ are not known yet. Some special cases to be investigated are as follows: Take $p=\frac 12$, $P_1$ as a uniform distribution on $C$, and $D=\set{\frac 23, \frac 56, 1}$. The optimal sets of $n$-means and the $n$th quantization errors for such a mixed distribution for all $n\geq 2$ are not known yet. Such a problem can also be investigated by taking  $P_1$ as a Cantor distribution, and $P_2$ discrete on $D$, for example, one can take $P_1$ the classical Cantor distribution, as considered in \cite{GL2}, and $D=\set{\frac 23, \frac 56, 1}$. Notice that $p$, $P_1$ and $D$ can be chosen in many different ways.
\end{rem}

\section{Quantization where $P_1$ and $P_2$ are Cantor distributions} \label{sec7}

Let $P_1$ be the Cantor distribution given by $P_1=\frac 12 P_1\circ S_1^{-1}+\frac 12 P_2\circ S_2^{-1}$, where $S_1(x)=\frac 13 x$ and $S_2(x)=\frac 13 x+\frac 29$ for all $x\in \D R$. Let $P_2$ be the Cantor distribution given by $P_2=\frac 12 P_2\circ T_1^{-1}+\frac 12 P_2\circ T_2^{-1}$, where $T_1(x)=\frac 14 x+\frac 12$ and $T_2(x)=\frac 14 x+\frac 34 $ for all $x\in \D R$. Let $C$ be the Cantor set generated by $S_1$ and $S_2$, and $D$ be the Cantor set generated by $T_1$ and $T_2$. Let $P$ be the mixed distribution generated by $P_1$ and $P_2$ such that $P=\frac 12 P_1+\frac 12 P_2$. Let $\set{1, 2}^\ast$ be the set of all words over the alphabet $\set{1, 2}$ including the empty word $\es$ as defined in Section~\ref{sec5}. Write $J:=[0, \frac 13]$ and $K:=[\frac 23, 1]$. Then, we have $C=\bigcap_{k\in \mathbb N} \bigcup_{\gs \in \{1, 2\}^k} J_\gs$ and $D=\bigcap_{k\in \mathbb N} \bigcup_{\gs \in \{1, 2\}^k} K_\gs$, where for $\gs\in \set{1, 2}^\ast$, $J_\gs=S_\gs([0, \frac 13])$ and $K_\gs=T_\gs([\frac 23, 1])$. Thus, $C$ is the support of $P_1$, and $D$ is the support of $P_2$ implying the fact that $C\uu D$ is the support of the mixed distribution $P$. As before, if nothing is mentioned within a parenthesis, by $\ga_n$ and $V_n$, we mean an optimal set of $n$-means and the corresponding quantization error for the mixed distribution $P$.

The following two lemmas are similar to Lemma~\ref{lemma52}.

\begin{lemma} Let $E(P_1)$ and $V(P_1)$ denote the expected value and the variance of a $P_1$-distributed random variable. Then, $E(P_1)=\frac 16$ and $V(P_1)=\frac 1{72}$. Moreover, for any $x_0 \in \D R$,
$\int (x-x_0)^2 \, dP_1 =V (P_1) +(x_0-\frac 16)^2.$
\end{lemma}

\begin{lemma}
Let $E(P_2)$ and $V(P_2)$ denote the expected value and the variance of a $P_2$-distributed random variable. Then, $E(P_2)=\frac 56$ and $V(P_2)=\frac{1}{60}$. Moreover, for any $x_0 \in \D R$,
$\int (x-x_0)^2 dP_2(x) =V (P_2) +(x_0-\frac 56)^2.$
\end{lemma}

We now prove the following lemma.

\begin{lemma} \label{lemma71}
Let $E(P)$ and $V(P)$ denote the expected value and the variance of a $P$-distributed random variable, where $P$ is the mixed distribution given by $P=\frac 12 P_1+\frac 12 P_2$. Then, $E(P)=\frac 12$ and $V(P)=\frac{91}{720}$. Moreover, for any $x_0 \in \D R$,
$\int (x-x_0)^2 dP(x) =V (P) +(x_0-\frac 12)^2.$
\end{lemma}
\nd\tbf{Proof:}  Let $X$ be a $P$-distributed random variable. Then,
\[E(X)=\int x dP(x)=\frac 12 \int x \, dP_1+\frac 12 \int x dP_2(x)=\frac 12\Big (\frac 16+\frac 56\Big)=\frac 12, \te{ and } \]
\[E(X^2)=\int x^2 dP(x)=\frac 12 \int x^2 \, dP_1+\frac 1 2 \int x^2 dP_2(x)=\frac 12 \Big(\frac{1}{24}+\frac{32}{45}\Big)=\frac{271}{720}, \]
and so, $V(P)=E(X^2)-(E(X))^2=\frac{91}{720}$. Then, by the standard theory of probability, for any $x_0 \in \D R$,
$\int (x-x_0)^2 dP(x) =V (P) +(x_0-\frac 12)^2.$ Thus, the proof of the lemma is complete.
\qed

\begin{rem}
From Lemma~\ref{lemma71}, it follows that the optimal set of one-mean for the mixed distribution $P$ is $\frac 12$ and the corresponding quantization error is $V(P)=\frac{91}{720}$. Again, notice that for any $x_0\in \D R$, we have
\[\int (x-x_0)^2 dP(x) =\frac 12 \Big (V (P_1)+V(P_2)+(x_0-\frac 16)^2+(x_0-\frac 56)^2\Big).\]
\end{rem}

\begin{defi}  \label{defi51} For $n\in \D N$ with $n\geq 2$, let $\ell(n)$ be the unique natural number with $2^{\ell(n)} \leq n<2^{\ell(n)+1}$. For $\gs \in \set{1, 2}^\ast$, let $a(\gs)$ and $b(\gs)$, respectively, denote the midpoints of the basic intervals $J_\gs$ and $K_\gs$. Let $I\sci \set{1, 2}^{\ell(n)}$ with card$(I)=n-2^{\ell(n)}$. Define $\gb_n(P_1, I)$ and $\gb_n(P_2, I)$ as follows:
\begin{align*}
\gb_n(P_1, I)&=\set{a(\gs) : \gs \in \set{1,2}^{\ell(n)} \setminus I} \uu \set{a(\gs 1) : \gs \in I} \uu \set {a(\gs2) : \gs \in I}, \te{ and } \\
\gb_n(P_2, I)&=\set{b(\gs) : \gs \in \set{1,2}^{\ell(n)} \setminus I} \uu \set{b(\gs 1) : \gs \in I} \uu \set {b(\gs2) : \gs \in I}.
\end{align*}

\end{defi}
The following proposition follows due to  \cite[Definition~3.5 and Proposition~3.7]{GL2}.

\begin{prop}  \label{prop51}  Let $\gb_n(P_1, I)$ and $\gb_n(P_2, I)$ be the sets for $n\geq 2$ given by Definition~\ref{defi51}. Then, $\gb_n(P_1, I)$ and $\gb_n(P_2, I)$ form optimal sets of $n$-means for $P_1$ and $P_2$, respectively, and the corresponding quantization errors are given by
\begin{align*}
V_n(P_1)&=\int\min_{a\in \gb_n(P_1, I)}\|x-a\|^2 \, dP_1=\frac 1{18^{\ell(n)}}\cdot\frac 1{72} \Big(2^{\ell(n)+1}-n+\frac 19\left(n-2^{\ell(n)}\right)\Big), \te{ and }  \\
V_n(P_2)&=\int\min_{a\in \gb_n(P_2, I)}\|x-a\|^2 \, dP_2=\frac 1{32^{\ell(n)}}\cdot\frac 1{60} \Big(2^{\ell(n)+1}-n+\frac 1{16}\left(n-2^{\ell(n)}\right)\Big).
\end{align*}
\end{prop}

\begin{prop} \label{prop10.0}
For $n\geq 2$, let $\ga_n$ be an optimal set of $n$-means for $P$. Then, $\ga_n\ii[0, \frac 13)\neq \es$ and $\ga_n\ii (\frac 23, 1]\neq \es$.
\end{prop}

\nd\tbf{Proof:}  Consider the set of two-points $\gb_2:=\set{\frac 16, \frac 56}$. Then,
\[\int \min_{a\in \gb_2}\|x-a\|^2dP=\frac 12 \Big(\int (x-\frac 16)^2 dP_1+\int(x-\frac 56)^2 dP_2\Big)=\frac{11}{720}=0.0152778.\]
Since $V_n$ is the quantization error for $n$-means for $n\geq 2$, we have $V_n\leq V_2\leq 0.0152778$.
Let $\ga_n=\set{a_1, a_2, a_3, \cdots, a_n}$ be an optimal set of $n$-means such that $a_1<a_2<a_3<\cdots<a_n$. Since the optimal points are centroids of their own Voronoi regions, we have $0<a_1<\cdots<a_n<1$. Assume that $\frac 13\leq a_1$. Then,
\[V_n\geq \int_{[0, \frac 13]}(x-\frac 13)^2 dP=\frac 12 \int_{[0, \frac 13]}(x-\frac 13)^2 dP_1=\frac{1}{48}=0.0208333>V_n,\]
which is a contradiction, and so we can assume that $a_1<\frac 13$. Next, assume that $a_n\leq \frac 23$. Then,
\[V_n\geq \int_{[\frac 23, 1]}(x-\frac 23)^2 dP=\frac 12 \int_{[\frac 23, 1]}(x-\frac 23)^2 dP_2=\frac{1}{45}=0.0222222>V_n,\]
which leads to a contradiction, and so we can assume that $\frac 23<a_n$. Thus, we see that $\ga_n\ii[0, \frac 13)\neq \es$ and $\ga_n\ii (\frac 23, 1]\neq \es$, which proves the proposition.
\qed

\begin{prop} \label{prop10.00}
For  $n\geq 2$, let $\ga_n$ be an optimal set of $n$-means for $P$. Then, $\ga_n$ does not contain any point from the open interval $(\frac 13, \frac 23)$. Moreover, the Voronoi region of any point from $\ga_n\ii J$ does not contain any point from $K$, and the Voronoi region of any point from $\ga_n\ii K$ does not contain any point from $J$.
\end{prop}

\nd\tbf{Proof:}  By Proposition~\ref{prop10.0}, the statement of the proposition is true for $n=2$. Now, we prove it for $n=3$.
Consider the set of three points $\gb_3:=\set{\frac 16, \frac{17}{24},\frac{23}{24}}$. Then,
\[\int \min_{a\in \gb_3}\|x-a\|^2dP=\frac 12\Big( \int_J (x-\frac 1{6})^2 dP_1+\int_{K_1} (x-\frac{17}{24})^2 dP_2+\int_{K_2}(x-\frac{23}{24})^2 dP_2\Big)=\frac{43}{5760}.\]
Since $V_3$ is the quantization error for three-means, we have $V_3\leq\frac{43}{5760}=0.00746528$. Let $\ga_3:=\set{a_1, a_2, a_3}$ be an optimal set of three-means such that $0<a_1<a_2<a_3<1$. By Proposition~\ref{prop10.0}, we have $a_1<\frac 13$ and $\frac 23<a_3$. Suppose that $a_2 \in (\frac 13, \frac 23)$. The following two cases can arise:

Case~1. $\frac 13<a_2\leq \frac 12$.

Then, $\frac 12(a_2+a_3)>\frac 23$ implying $a_3>\frac 43-a_2\geq \frac 43-\frac 12=\frac 56$. Using an equation similar to \eqref{eq234}, we can show that for $\frac 56< a_3<1$, the error $\frac 12 \int_{K}(x-a_3)^2dP_2$ is minimum if $P$-almost surely, $a_3=\frac 56$, and the minimum value is $\frac{1}{120}$. Thus,
\[V_3\geq \frac 12 \int_K(x-\frac 56)^2 dP_2=\frac{1}{120}=0.00833333>V_3,\]
which is a contradiction.

Case~2. $\frac 12\leq a_2<\frac 23$.

Then, $\frac 12(a_1+a_2)<\frac 13$ implying $a_1<\frac 23-a_2\leq \frac 23-\frac 12=\frac 16$. Similar in Case~1, for $0<a_1<\frac 16$, the error $\frac 12 \int_{J}(x-a_1)^2dP_1$ is minimum if $P$-almost surely, $a_1=\frac 16$, and the minimum value is $\frac{1}{144}$.
Thus,
\[V_3\geq \frac 1{144}+\frac 12\int_{K_1}(x-\frac 23)^2dP_2=\frac{11}{1440}=0.00763889>V_3,\]
which leads to a contradiction.

Thus, by Case~1 and Case~2, we see that $\ga_3$ does not contain any point from $(\frac 13, \frac 23)$. We now prove the proposition for all $n\geq 4$. Consider the set of four points $\gb_4:=\set{\frac 1{18}, \frac 5{18},  \frac{17}{24}, \frac{23}{24}}$. The distortion error due to the set $\gb_4$ is given by
\[\int\min_{a\in \gb_4}\|x-a\|^2 dP=\frac 12(V_2(P_1)+V_2(P_2))=\frac{67}{51840}=0.00129244.\]
Since $V_n$ is the quantization error for $n$-means for all $n\geq 4$, we have $V_n\leq V_4\leq 0.00129244$. Let $j=\max\set{i : a_i<\frac 23 \te{ for all } 1\leq i\leq n}$. Then, $a_j<\frac 23$. We need to show that $a_j<\frac 13$. For the sake of contradiction, assume that $a_j\in (\frac 13, \frac 23)$. Then, two cases can arise:

Case~A. $\frac 13<a_j\leq \frac 12$.

Then, $\frac 12 (a_j+a_{j+1})>\frac 23$ implying $a_{j+1}>\frac 43-a_j\geq \frac 43-\frac 12=\frac 56$, and so,
\[V_n\geq \frac 12 \int_{K_1}(x-\frac 56)^2 dP_2=\frac{1}{240}=0.00416667>V_n,\]
which leads to a contradiction.

Case~B. $\frac 12\leq a_j\leq \frac 23$.

Then, $\frac 12 (a_{j-1}+a_j)<\frac 13$ implying $a_{j-1}<\frac 23-a_j\leq \frac 23-\frac 12=\frac 16$, and so,
\[V_n\geq \frac 12 \int_{J_2}(x-\frac 16)^2 dP_1=\frac{1}{288}=0.00347222>V_n,\]
which gives a contradiction.

Thus, by Case~A and Case~B, we can assume that $a_j\leq \frac 13$. If the Voronoi region of any point from $\ga_n\ii J$ contains points from $K$, then we must have $\frac 12(a_j+a_{j+1})>\frac 23$ implying $a_{j+1}>\frac 43-a_j\geq \frac 43-\frac 13=1$, which is a contradiction since $a_{j+1}<1$. Similarly, the Voronoi region of any point from $\ga_n\ii K$ does not contain any point from $J$. Thus, the proof of the proposition is complete.
\qed

\begin{note}
 From Proposition~\ref{prop10.0} and Proposition~\ref{prop10.00}, it follows that for $n\geq 2$, if an optimal set $\ga_n$ contains $n_1$ elements from $J$ and $n_2$ elements from $K$, then $n=n_1+n_2$. In that case, we write $\ga_n:=\ga_{(n_1, n_2)}$ and $V_n:=V_{(n_1, n_2)}$. Thus, $\ga_n=\ga_{(n_1, n_2)}=\ga_{n_1}(P_1)\uu \ga_{n_2}(P_2)$, and $V_n=V_{(n_1, n_2)}=\frac 12(V_{n_1}(P_1)+V_{n_2}(P_2))$.
\end{note}

\begin{lemma} \label{lemma10.1}
Let $\ga$ be an optimal set of two-means for $P$. Then, $\ga=\ga_{(1,1)}$, and the corresponding quantization error is $V_2=\frac{5}{432}=0.0115741$.
\end{lemma}

\nd\tbf{Proof:}
Let $\ga=\set{a_1, a_2}$ be an optimal set of two-means such that $0<a_1<a_2<1$. By Proposition~\ref{prop10.0}, we have $a_1<\frac 13$ and $\frac 23<a_2$ yielding $a_1=\frac 16$, $a_2=\frac 56$, i.e., $\ga=\ga_1(P_1)\uu \ga_1(P_2)$, and $V_2=\frac{11}{720}=0.0152778$. Thus, the proof of the lemma is complete.
\qed

\begin{lemma} \label{lemma10.2}
Let $\ga$ be an optimal set of three-means. Then, $\ga=\ga_{(1, 2)}$, and the corresponding quantization error is $V_3=\frac{43}{5760}=0.00746528$.
\end{lemma}
\nd\tbf{Proof:}
Let $\ga$ be an optimal set of three-means. By Proposition~\ref{prop10.0} and Proposition~\ref{prop10.00}, we can assume that either $\ga=\ga_2(P_1) \uu \ga_1(P_2)$, or $\ga=\ga_1(P_1) \uu \ga_2(P_2)$. Since
\[\int\min_{a\in \ga_1(P_1) \uu \ga_2(P_2)}(x-a)^2 dP<\int\min_{a\in \ga_2(P_1) \uu \ga_1(P_2)}(x-a)^2 dP,\]
the set $\ga=\ga_1(P_1) \uu \ga_2(P_2)$ forms an optimal set of three-means, and the corresponding quantization error is
\[V_3=\int\min_{a\in \ga_1(P_1) \uu \ga_2(P_2)}(x-a)^2 dP=\frac12(V_1(P_1)+V_2(P_2))=\frac{43}{5760}=0.00746528,\]
which yields the lemma.
\qed
\begin{lemma} \label{lemma10.3}
Let $\ga$ be an optimal set of four-means. Then, $\ga=\ga_{(2,2)}$, and the corresponding quantization error is $V_4=\frac{67}{51840}=0.00129244$.
\end{lemma}
\nd\tbf{Proof:}
Let $\ga$ be an optimal set of four-means. By Proposition~\ref{prop10.0} and Proposition~\ref{prop10.00}, we can assume that either $\ga=\ga_3(P_1) \uu \ga_1(P_2)$, $\ga=\ga_2(P_1) \uu \ga_2(P_2)$, or $\ga=\ga_1(P_1) \uu \ga_3(P_2)$ . Among all these possible choices, we see that $\ga=\ga_2(P_1) \uu \ga_2(P_2)$ gives the minimum distortion error, and hence, $\ga=\ga_2(P_1) \uu \ga_2(P_2)$ is an optimal set of four-means, and the corresponding quantization error is
$V_4=\frac 12(V_2(P_1)+V_2(P_2))=\frac{67}{51840}=0.00129244$,
which is the lemma.
\qed

\begin{rem} \label{rem1001}  Proceeding in the similar way, as Lemma~\ref{lemma10.3}, it can be proved that the optimal sets of $n$-means for $n=5,6,7,$ etc. are, respectively, $\ga_{(3,2)}, \, \ga_{(2^2, 2)} \, \ga_{(2^2, 3)}$, etc.
\end{rem}

We now prove the following lemma.

\begin{lemma} \label{lemma11.0}
 Let $\ga_{(2^{6n-4}, 2^{5n-4})}$ be an optimal set of $2^{6n-4}+2^{5n-4}$-means for $P$ for some positive integer $n$. For $1\leq i\leq 5$ and $1\leq j\leq 6$, let $\ell_i, k_j \in \D N$ be such that $1\leq \ell_i\leq 2^{5n-4+(i-1)}$ and $1\leq k_j\leq 2^{6n-4+(j-1)}$.
 Then,
$(i)$ $\ga_{(2^{6n-4}, 2^{5n-4}+\ell_1)}$ is an optimal set of $2^{6n-4}+2^{5n-4}+\ell_1$-means;  $(ii)$ $\ga_{(2^{6n-4}+k_1, 2^{5n-3})}$ is an optimal set of  $2^{6n-4}+2^{5n-3}+k_1$-means; $(iii)$ $\ga_{(2^{6n-3}, 2^{5n-3}+\ell_2)}$ is an optimal set of $2^{6n-3}+2^{5n-3}+\ell_2$-means; $(iv)$ $\ga_{(2^{6n-3}+k_2, 2^{5n-2})}$ is an optimal set of $2^{6n-3}+2^{5n-2}+k_2$-means; $(v)$ $\ga_{(2^{6n-2}, 2^{5n-2}+\ell_3)}$ is an optimal set of $2^{6n-2}+2^{5n-2}+\ell_3$-means; $(vi)$ $\ga_{(2^{6n-2}+k_3, 2^{5n-1})}$ is an optimal set of $2^{6n-2}+2^{5n-1}+k_3$-means; $(vii)$ $\ga_{(2^{6n-1}, 2^{5n-1}+\ell_4)}$ is an optimal set of $2^{6n-1}+2^{5n-1}+\ell_4$-means; $(viii)$ $\ga_{(2^{6n-1}+k_4, 2^{5n})}$ is an optimal set of $2^{6n-1}+2^{5n}+k_4$-means; $(ix)$ $\ga_{(2^{6n}, 2^{5n}+\ell_5)}$ is an optimal set of $2^{6n}+2^{5n}+\ell_5$-means;
$(x)$ $\ga_{(2^{6n}+k_5, 2^{5n+1})}$ is an optimal set of $2^{6n}+2^{5n+1}+k_5$-means; and $(xi)$ $\ga_{(2^{6n+1}+k_6, 2^{5n+1})}$ is an optimal set of $2^{6n+1}+2^{5n+1}+k_6$-means.
\end{lemma}
\nd\tbf{Proof:}  By Remark~\ref{rem1001}, it is known that $\ga_{(2^{6n-4}, 2^{5n-4})}$ is an optimal set of $2^{6n-4}+2^{5n-4}$-means for $n=1$. So, we can assume that $\ga_{(2^{6n-4}, 2^{5n-4})}$ is an optimal set of $2^{6n-4}+2^{5n-4}$-means for $P$ for some positive integer $n$. Recall that $\ga_{(n_1, n_2)}$ is an optimal set of $n_1+n_2$-means, and contains $n_1$ elements from $C$ and $n_2$ elements from $D$, and so, an optimal set of $n_1+n_2+1$-means must contain at least $n_1$ elements from $C$, and at least $n_2$ elements from $D$. For all $n\geq 1$, since
\[\frac 12(V_{2^{6n-4}}(P_1)+V_{2^{5n-4}+1}(P_2))<\frac 12(V_{2^{6n-4}+1}(P_1)+V_{2^{5n-4}}(P_2)),\]
we can assume that $\ga_{(2^{6n-4}, 2^{5n-4}+\ell_1)}$ is an optimal set of $2^{6n-4}+2^{5n-4}+\ell_1$-means for $\ell_1=1$. Having known  $\ga_{(2^{6n-4}, 2^{5n-4}+1)}$ as an optimal set of $2^{6n-4}+2^{5n-4}+1$-means, we see that
\[\frac 12(V_{2^{6n-4}}(P_1)+V_{2^{5n-4}+2}(P_2))<\frac 12(V_{2^{6n-4}+1}(P_1)+V_{2^{5n-4}+1}(P_2)),\]
and so, $\ga_{(2^{6n-4}, 2^{5n-4}+\ell_1)}$ is an optimal set of $2^{6n-4}+2^{5n-4}+\ell_1$-means for $\ell_1=2$. Proceeding in this way, inductively, we can show that $\ga_{(2^{6n-4}, 2^{5n-4}+\ell_1)}$ is an optimal set of $2^{6n-4}+2^{5n-4}+\ell_1$-means for $1\leq \ell_1\leq 2^{5n-4}$. Thus, $(i)$ is true. Now, by $(i)$, we see that $\ga_{(2^{6n-4}, 2^{5n-3})}$ is an optimal set of $2^{6n-4}+2^{5n-3}$-means. Then, proceeding in the same way as $(i)$ we can show that $(ii)$ is true. Similarly, we can prove the statements from $(iii)$ to $(xi)$. Thus, the lemma is yielded.
\qed

\begin{prop} \label{prop12.0}  The sets $\ga_{(2^{6n-4}, 2^{5n-4})}$,  $\ga_{(2^{6n-4}, 2^{5n-3})}$,  $\ga_{(2^{6n-3}, 2^{5n-3})}$, $\ga_{(2^{6n-3}, 2^{5n-2})}$, $\ga_{(2^{6n-2}, 2^{5n-2})}$, $\ga_{(2^{6n-2}, 2^{5n-1})}$, $\ga_{(2^{6n-1}, 2^{5n-1})}$, $\ga_{(2^{6n-1}, 2^{5n})}$, $\ga_{(2^{6n}, 2^{5n})}$, $\ga_{(2^{6n}, 2^{5n+1})}$, $\ga_{(2^{6n+1}, 2^{5n+1})}$, and $\ga_{(2^{6n+2}, 2^{5n+1})}$  are optimal sets for all $n\in \D N$.
\end{prop}

\nd\tbf{Proof:}   By Remark~\ref{rem1001}, it is known that $\ga_{(2^{6n-4}, 2^{5n-4})}$ is an optimal set of $2^{6n-4}+2^{5n-4}$-means for $n=1$. Then, by Lemma~\ref{lemma11.0}, it follows that $\ga_{(2^{6n-4}, 2^{5n-4})}$ is an optimal set of $2^{6n-4}+2^{5n-4}$-means for $n=2$, and so, applying Lemma~\ref{lemma11.0} again, we can say that $\ga_{(2^{6n-4}, 2^{5n-4})}$ is an optimal set of $2^{6n-4}+2^{5n-4}$-means for $n=3$. Thus, by induction, $\ga_{(2^{6n-4}, 2^{5n-4})}$ are optimal sets of $2^{6n-4}+2^{5n-4}$-means for all $n\geq 2$. Hence, by Lemma~\ref{lemma11.0}, the statement of the proposition is true.
\qed

\begin{rem} Because of Lemma~\ref{lemma71}, Lemma~\ref{lemma10.1}, Lemma~\ref{lemma10.2}, Lemma~\ref{lemma10.3}, and Remark~\ref{rem1001}, the optimal sets of $n$-means are known for all $1\leq n\leq 6$. To determine the optimal sets of $n$-means for any $n\geq 6$, let $\ell(n)$ be the least positive integer such that
$2^{6\ell(n)-4}+2^{5\ell(n)-4}\leq n<2^{6(\ell(n)+1)-4}+ 2^{5(\ell(n)+1)-4}$. Then, using Lemma~\ref{lemma11.0}, we can determine $n_1$ and $n_2$ with $n=n_1+n_2$ so that $\ga_n=\ga_{(n_1, n_2)}$ gives an optimal set of $n$-means. Once $n_1$ and $n_2$ are known, the corresponding quantization error is obtained by using the formula $V_n=\frac 12(V_{n_1}(P_1)+V_{n_2}(P_2))$.
\end{rem}

\subsection{Asymptotics for the $n$th quantization error $V_n(P)$}
In this subsection, we investigate the quantization dimension and the quantization coefficients for the mixed distribution $P$. Let $\gb_1$ be the Hausdorff dimension of the Cantor set $C$ generated by $S_1$ and $S_2$, and $\gb_2$ be the Hausdorff dimension of the Cantor set $D$ generated by $T_1$ and $T_2$. Then, $
\gb_1=\frac{\log 2}{\log 3}$ and $\gb_2=\frac 12$. If $D(P_i)$ are the quantization dimensions of $P_i$ for $i=1, 2$, then it is known that $D(P_1)=\gb_1$ and $D(P_2)=\gb_2$ \cite{GL2}. By Proposition~\ref{prop1000}, the following theorem is true.

\begin{theo1} \label{Th2}
Let $D(P)$ be the quantization dimension of the mixed distribution $P:=\frac 12 P_1+\frac 12 P_2$. Then, $D(P)=\max \set{D(P_1), D(P_2)}$.
\end{theo1}

\begin{theo1}\label{Th3}
Quantization coefficient for the mixed distribution $P:=\frac 12 P_1 +\frac 12 P_2$ does not exist.
\end{theo1}

\nd\tbf{Proof:}  By Theorem~\ref{Th2}, the quantization dimension of the mixed distribution exists and equals $\gb_1$, where $\gb_1=\frac {\log 2}{\log 3}$. To prove the theorem it is enough to show that the sequence $\Big(n^{\frac 2{\gb_1}}V_n(P)\Big)_{n\geq 1}$ has at least two different accumulation points. By Lemma~\ref{lemma11.0} (i), it is known that $\ga_{(2^{6n-4}, 2^{5n-4})}$ is an optimal set of $2^{6n-4}+2^{5n-4}$-means. Again, by Lemma~\ref{lemma11.0} (ii), it is known that $\ga_{(2^{6n-4}+2^{6n-5}, 2^{5n-3})}$ is an optimal set of $2^{6n-4}+2^{6n-5}+2^{5n-3}$-means. Write $F(n):=2^{6n-4}+2^{5n-4}$, and $G(n):=2^{6n-4}+2^{6n-5}+2^{5n-3}$ for $n\in \D N$. Recall that
\begin{align*}
V_{F(n)}&=V_{(2^{6n-4}, 2^{5n-4})}=\frac 12\Big(V_{2^{6n-4}}(P_1)+V_{2^{5n-4}}(P_2)\Big)=\frac{1}{240} \left(2^{17-20 n}+5\cdot 3^{7-12 n}\right), \\
V_{G(n)}&=V_{(2^{6n-4}+2^{6n-5}, 2^{5n-3})}=\frac 12\Big(V_{2^{6n-4}+2^{6n-5}}(P_1)+V_{2^{5n-3}}(P_2)\Big)=\frac{1}{15} 2^{9-20 n}+\frac{5}{16} 81^{1-3 n}.
\end{align*}
Notice that $(2^{6n})^{\frac 2{\gb_1}}=2^{\frac {12n\log 3}{\log 2}}=3^{12n}$ and $\lim_{n\to \infty } \left(\frac{3^{12}}{2^{20}}\right)^n=0$, and so, we have
\begin{align*}
&\lim_{n\to \infty} F(n)^{\frac 2 {\gb1}}V_{F(n)}(P) =\lim_{n\to \infty} (2^{6n-4}+2^{5n-4})^{\frac 2{\gb_1}}\frac{1}{240} \left(2^{17-20 n}+5\cdot 3^{7-12 n}\right)\\
&=\lim_{n\to\infty} 3^{12n}\Big(\frac 1 {2^4}+\frac 1{2^4}\cdot \frac 1{2^n}\Big)^{\frac 2{\gb_1}}\frac{1}{240} \left(2^{17-20 n}+5\cdot 3^{7-12 n}\right)=2^{-\frac{8}{\gb_1}}\frac {5\cdot 3^7}{240}=\frac{1}{144}=0.00694444,
\end{align*}
and
\begin{align*}
&\lim_{n\to \infty} G(n)^{\frac 2 {\gb1}}V_{G(n)}(P) =\lim_{n\to \infty} (2^{6n-4}+2^{6n-5}+2^{5n-3})^{\frac 2{\gb_1}}(\frac{1}{15}\cdot 2^{9-20 n}+\frac{5}{16}\cdot 81^{1-3 n})\\
&=\lim_{n\to \infty}3^{12n}\Big(\frac 1{2^4}+\frac 1{2^5}+\frac 1{2^3}\frac 1{2^n}\Big)^{\frac 2{\gb_1}}(\frac{1}{15}\cdot 2^{9-20 n}+\frac{5}{16}\cdot 81 \cdot 3^{-12n})=\frac{5}{16}\cdot 3^{\frac{2 \log (3)}{\log (2)}-6}=0.0139496.
\end{align*}
Since $(F(n)^{\frac 2 {\gb1}}V_{F(n)}(P))_{n\geq 1}$ and $( G(n)^{\frac 2 {\gb1}}V_{G(n)}(P))_{n\geq 2}$ are two subsequences of $(n^{\frac 2{\gb_1}}V_n(P))_{n\in \D N}$ having two different accumulation points, we can say that the sequence $(n^{\frac 2{\gb_1}}V_n(P))_{n\in \D N}$ does not converge, in other words, the $\gb_1$-dimensional quantization coefficient for $P$ does not exist. This completes the proof of the theorem.
\qed

We now conclude the paper with the following section.

\section{Discussion and open problems} \label{sec8}

Let $P_1$ and $P_2$ be two uniform distributions defined on the base $L_1:=\set{(t, 0) : -1\leq t\leq 1}$, and the semicircular arc $L_2:=\set{(\cos t, \sin t) : 0\leq t\leq \pi}$ of the semicircular disc $x_1^2+x_2^2=1$, where $x_2\geq 0$. Write $P:=p_1P_1+p_2P_2$, where $(p_1, p_2)$ is a probability vector. Then, $P$ is a mixed distribution with support $L:=L_1\uu L_2$. The determination of the optimal sets of $n$-means for smaller values of $n$ for such a mixed distribution is not so difficult, but for the higher values of $n$ it needs extensive work. If we know how many points in an optimal set $\ga_n$ of $n$-means for $P$ are coming from $L_1$ such that the Voronoi region of any point of which does not contain any point from $L_2$, or  how many points are coming due to $L_2$ such that the Voronoi region of any point of which does not contain any point from $L_1$, then we can easily determine the optimal set $\ga_n$ for $P$. Set $p_1=p_2=\frac 12$, i.e., take $P=\frac 12 P_1+\frac 12P_2$.

\begin{defi} \label{difi21}
Define the sequence $\set{a(n)}$ such that $a(n)=\lfloor n(\sqrt 2-1)\rfloor$ for $n\geq 1$, i.e.,
\begin{align*}
 \set{a(n)}_{n=1}^\infty=&\set{0, 0, 1, 1,2,2,2,3,3,4,4,4,5,5,6,6,7,7,7,8,8,9,9,9,10,10,11,11,12,12, \cdots},
\end{align*}
where $\lfloor x\rfloor$ represents the greatest integer not exceeding $x$.
\end{defi}

\begin{algorithm}  Let $V(n, k)$ be the distortion error if we assume that an optimal set $\ga_n$ of $n$-means for $P$ contains $k$-elements from $L_1$ the Voronoi region of any point of which does not contain any point from $L_2$. Let $\set{a(n)}$ be the sequence  defined by Definition~\ref{difi21}. Define an algorithm  as follows:

$(i)$ Write $k:=a(n)$.

$(ii)$ If $k=1$ go to step $(v)$, else step $(iii)$.

$(iii)$ If $V(n, k-1)<V(n, k)$ replace $k$ by $k-1$ and go to step $(ii)$, else step $(iv)$.

$(iv)$ If $V(n, k+1)<V(n, k)$ replace $k$ by $k+1$ and return, else step $(v)$.

$(v)$ End.
\end{algorithm}
When the algorithm ends, then the value of $k$, obtained, is the actual value of $k$ that an optimal set $\ga_n$ for $P$ contains from the base $L_1$ of the semicircular disc. For example, if $n=5000$, then $a(n)=2071$, and by running the algorithm we obtain $k=2083$. This tells us that an optimal set $\ga_n$ of $n$-means for $P$ contains $2083$ elements from $L_1$ the Voronoi region of any point of which does not contain any point from $L_2$, $2$ elements from the interior of the angles formed by the base $L_1$ and the semicircular arc $L_2$, the Voronoi regions of these two points contain points from both $L_1$ and $L_2$, and the remaining $5000-2083-2$ points are from $L_2$ the Voronoi region of any point of which does not contain any point from $L_1$. Thus, we see that the above sequence and the algorithm help us to correctly determine an optimal set of $n$-means for the mixed distribution $P=\frac 12 P_1+\frac 12 P_2$. For the details of it see \cite{PRRSS}.
For any other probability vector $(p_1, p_2)$ what will be the sequence is not known yet, i.e., a general formula to determine the optimal set of $n$-means for the mixed distribution $P=p_1P_1+p_2P_2$, where $P_1$ and $P_2$ are two uniform distributions as defined before, is not known yet. In fact, one can fix a probability vector $(p_1, p_2)$, and vary the probability measures $P_1$ and $P_2$ to investigate the optimal sets of $n$-means for the mixed distribution $P$ for any positive integer $n$. For example, let us take $P:=\frac 12 P_1+\frac 12 P_2$, where $P_1$ is a uniform distribution with support two perpendicular diameters of a circle, and $P_2$ is a uniform distribution defined on the circle. Notice that optimal quantizations for $P_1$ and $P_2$, in this case, are already known, but for the mixed distribution $P$ the optimal sets of $n$-means, and the $n$th quantization errors  for all positive integers $n$ are not known yet.

Optimal quantization for a general probability measure, singular or nonsingular, is still open, which yields the fact that the optimal quantization for a mixed distribution taking any two probability measures is not yet known. The results in our paper, will further motivate the interested researchers to investigate the optimal quantization for more general mixed distributions.

\end{document}